\documentclass[11pt]{amsart}
\usepackage {amsmath, amssymb,   epsfig, a4wide, enumerate, psfrag}
\usepackage[utf8]{inputenc}
\usepackage[english]{babel}

\date{\today}

\author{Romain Dujardin}
\thanks{Research  partially supported by ANR project LAMBDA,  ANR-13-BS01-0002 and  a grant from the  Institut Universitaire de France}
 
\title{Non-density of stability  for holomorphic mappings on~$\pk$}

\address{Laboratoire de probabilités et modèles aléatoires, UMR 7599, Université Pierre et Marie Curie, 4 place Jussieu, 75005 Paris, France}
\email{romain.dujardin@upmc.fr}


\newcommand{\cc}{\mathbb{C}}
\newcommand{\bb}{\mathbb{B}}
\newcommand{\re}{\mathbb{R}}
\newcommand{\dd}{\mathbb{D}}

\newcommand{\nn}{\mathbb{N}}
\newcommand{\pp}{\mathbb{P}}
\newcommand{\e}{\varepsilon}
\newcommand{\E}{\mathcal{E}}
\newcommand{\F}{\mathcal{F}}

 \newcommand{\cv}{\rightarrow}

\newcommand{\fr}{\partial}
\newcommand{\om}{\Omega}
\newcommand{\set}[1]{\left\{#1\right\}}
\newcommand{\norm}[1]{\left\Vert#1\right\Vert}
\newcommand{\abs}[1]{\left\vert#1\right\vert}
\newcommand{\cd}{{\cc^2}}
\newcommand{\pd}{{\mathbb{P}^2}}
\newcommand{\pu}{{\mathbb{P}^1}}
\newcommand{\pk}{{\mathbb{P}^k}}

\newcommand{\rest}[1]{ \arrowvert_{#1}}

\newcommand{\unsur}[1]{\frac{1}{#1}}
\newcommand{\el}{\mathcal{L}}

\newcommand{\rond}{\hspace{-.1em}\circ\hspace{-.1em}}
\newcommand{\cst}{\mathrm{C}^\mathrm{st}}
\newcommand{\lrpar}[1]{\left(#1\right)}

\newcommand{\la}{\lambda}
\newcommand{\lo}{{\lambda_0}}

\newcommand{\La}{\Lambda}

\newcommand{\tbif}{T_{\mathrm{bif}}}

\newcommand{\hold}{\mathcal{H}_d}
\newcommand{\hol}{\mathcal{H}}
\newcommand{\poly}{\mathcal{P}}

\newcommand{\inv}{^{-1}}

\DeclareMathOperator{\supp}{Supp}

\DeclareMathOperator{\bif}{Bif}

\DeclareMathOperator{\crit}{Crit}
\DeclareMathOperator{\id}{id}

\newtheorem{prop}{Proposition} [section]
\newtheorem{thm}[prop] {Theorem}
\newtheorem{defi}[prop] {Definition}
\newtheorem{lem}[prop] {Lemma}
\newtheorem{cor}[prop]{Corollary}

\newtheorem{propdef}[prop]{Proposition-Definition}

\newtheorem{conjecture}[prop]{Conjecture}

\newtheorem*{mainthm}{Main Theorem}

\theoremstyle{remark}

\newtheorem{rmk}[prop]{Remark}

\begin{document}
 
\begin{abstract} 
A well-known theorem due to Mañé-Sad-Sullivan and Lyubich asserts that $J$-stable maps are dense in any holomorphic family 
of rational maps in dimension 1. In this paper we show that the corresponding result fails in higher dimension. More precisely, we construct open subsets 
in the bifurcation locus in the space of holomorphic mappings of degree $d$ of $\mathbb{P}^k(\mathbb{C})$ for every $d\geq 2$ and $k\geq 2$. 
\end{abstract}

\maketitle

\section{Introduction}
  
  Let $(f_\la)_{\la\in \La}$ be a holomorphic family of holomorphic self-mappings on the complex projective space $\pk$. 
  When the dimension $k$ equals 1, the stability/bifurcation theory of such families was 
developed in the beginning of the 1980's independently by  Mañé, Sad and Sullivan \cite{mss} and Lyubich
 \cite{lyubich bif, lyubich bif2},  who designed the  seminal notion of   $J$-stability, 
 that is, structural stability on the Julia set, which almost implies 
structural stability on $\pu$ (see \cite{mcms}). A salient feature of their work  
 is that the set of $J$-stable mappings is (open and) dense in any such family, which  is ultimately   
  a consequence of the finiteness of the critical set. 

\medskip
 
 In higher dimension we denote by $J^*$   the ``small Julia set" of $f$, which by definition  is the support of its measure of maximal entropy.  
 Typically, $J^*$ is smaller than the usual Julia set $J$, and  
   concentrates  in a sense the repelling part of the dynamics of $f$, and most of the entropy. 
    
 In a  remarkable recent paper,    Bianchi, Berteloot and Dupont \cite{bbd} 
put forward  $J^*$-stability as the correct generalization of 
 the Mañé-Sad-Sullivan-Lyubich theory  in higher dimension. 
   As these authors point out, there is no reason to expect that the density of $J^*$-stability should hold in this setting, and 
   leave the existence  of persistent bifurcations as an open problem \cite[\S 6.2]{bbd}. 
   
   \medskip
   
   At this stage it is worth mentioning that for  invertible  dynamics in dimension 2,    
   the classical work  of Newhouse \cite{newhouse} 
   shows  that   persistent bifurcations can be obtained by constructing persistent (generic) homoclinic tangencies. In the holomorphic context, 
 Buzzard \cite{buzzard} proved that  persistent homoclinic tangencies exist in the space of polynomial automorphisms of $\cd$
  of sufficiently high degree.  Furthermore such automorphisms  can easily   be ``embedded''
   inside holomorphic mappings of $\pd$ (see e.g. \cite[\S 6]{hp} and  also  \cite{gavosto}).  
It seems however that for these examples, the Newhouse phenomenon is somehow unrelated to   $J^*$-(in)stability, which has to do with repelling periodic orbits
 (see \S\ref{subs:newhouse} for a more precise discussion).     
   
   \medskip

  Our main goal in this  paper is to address the density problem     for $J^*$-stability.
   We denote by $\mathcal{H}_d(\pk)$ the space of holomorphic mappings of $\pk$ of degree $d$.
 By choosing a set of homogeneous coordinates on $\pk$, we can express 
  $f\in \hold(\pk)$ in terms of a family of $(k+1)$   homogeneous polynomials in $(d+1)$ variables 
  without common factor. In this way we get a   natural  identification between 
  $\mathcal{H}_d(\pk)$   and a Zariski open subset of   $\pp^N$, for  $N= (k+1)\frac{(d+k)!}{d!k!}-1$.  
  
  \begin{mainthm}
  The bifurcation locus   has  non-empty interior in the space $\mathcal{H}_d(\pk)$   for every $k\geq 2$ and $d\geq 2$. 
  \end{mainthm}

  As we shall see, our constructions
   are rather specific, since we  start with the simplest possible mappings, namely products, and construct robustly bifurcating examples by taking
    small perturbations.  More interestingly, we isolate some {\em mechanisms}   leading to 
   open sets of  bifurcations. The next step would be to understand the prevalence of these phenomena in parameter space. 
   
   While he was  working on this paper, the author learned 
   about the work of  Bianchi and Taflin \cite{bianchi taflin}, in which the authors 
    construct a  1-dimensional family of rational maps on $\pd$  whose   bifurcation locus is the whole family, thus offering 
    an alternative approach to the question of \cite{bbd}.
   
   Throughout the paper, to avoid overwhelming the ideas under undue technicalities, we  
    focus on dimension $k=2$ which is the first interesting case, and explain in \S\ref{subs:higher} 
    how to adapt the arguments to  higher  dimension. 
    
  \medskip
  
  By the analysis of Berteloot, Bianchi and Dupont, to obtain open subsets in the  bifurcation locus, it is enough to create a persistent intersection  
 between the post-critical set and a hyperbolic repeller    contained in  $J^*$. We present two mechanisms leading to such persistent intersections. 
The first one is based on topological considerations:
the idea is to construct a kind of topological manifold contained in $J^*$, which must intersect the 
post-critical set for homological reasons.  
  We apply this strategy for mappings that are small perturbations of product maps of the form 
$(p(z), w^d)$ (see Theorem \ref{thm:topology}).  Of course what  is delicate here is to manage  
dynamically defined sets with inherently  complicated topology. 

\medskip
 
The second mechanism is based on ideas from fractal geometry. 
It relies on the construction of   Cantor sets 
  with a very special geometry: namely they are ``fat" in a certain direction and 
  admit  persistent intersections with  manifolds that are  sufficiently transverse
to the fat direction. 
   The use of this property in dynamical systems originates in the  work of Bonatti and Diaz \cite{bonatti diaz} and has played an important role in $C^1$ dynamics 
   since then  (see the book by Bonatti, Diaz and Viana
\cite{bonatti diaz viana} for a broader perspective).  According to the  real dynamics  terminology we will    refer to these Cantor sets as {\em blenders}. 

We implement this idea for product mappings of the form $(p(z), w^d+\kappa)$, where $p$ admits a repelling fixed point 
$z_0$  which is almost parabolic, that is $1<\abs{p'(z_0)}< 1.01$, and $\kappa$ is large (see Theorem \ref{thm:blender}). The  parabolic case 
  follows by taking the limit (see Corollary \ref{cor:parabolic}). 

    \medskip
 
  Despite the specific nature of these examples, our thesis is that the interior of the bifurcation locus is quite large. We point out a few explicit questions and conjectures in \S\ref{subs:open}.   In particular since by \cite{bbd} bifurcations happen when a multiplier of a repelling periodic point crosses the unit circle, 
  it is likely that blenders are often created  in this process. 
   
Note also that  blenders already appear  in complex dynamics 
   in the work of Biebler \cite{biebler} to construct polynomial automorphisms of $\cc^3$ with persistent homoclinic tangencies.

\medskip

The plan of the paper is as follows. In  \S\ref{sec:prel} we collect a few facts from $J^*$-stability theory. 
 The topological mechanism for robust bifurcations is detailed
 in \S \ref{sec:topology}, while  \S\ref{sec:blender} is devoted to complex  blenders. 
 Finally in \S\ref{sec:further} we explain how to extend the construction of \S \ref{sec:blender} 
 to higher dimensions and also show that $J^*$-stability is compatible with the Newhouse phenomenon.  
 We also state a number of open problems and directions for future research. The main theorem ultimately follows from Theorems \ref{thm:topology} and Theorem \ref{thm:blender} for the case $d=2$, $k\geq 3$, Theorem \ref{thm:blender2} for   $d=2$, $k=2$ and Theorem \ref{thm:higherdim} for $k\geq 3$. 
 
 Thanks to Fabrizio Bianchi and Lasse Rempe-Gillen for useful comments. 
 
  \section{Preliminaries}\label{sec:prel}
  
 {\bf  Notation:}
 when a preferred parameter $\lo$ is given in $\La$, 
 to simplify notation we denote $f_\lo$ by $f_0$. 
We use  the convention to mark with a hat the objects  in $\La\times \pk$, like $\widehat{f}:(\la, z)\mapsto (\la, f_\la(z))$, etc.

   \subsection{Stability and bifurcations for endomorphisms of $\pk$}
   Let $f$ be a holomorphic map of degree $d$ on $\pk$. A classical result of Briend and Duval asserts that it admits a 
      unique invariant measure of maximal entropy $\mu_f$, whose (complex)
    Lyapunov exponents  $\chi_i$, $i =1, \ldots , k$ satisfy $\chi_i\geq \unsur{2} \log d$, and 
    which describes the asymptotic distribution of (resp. repelling) 
    periodic orbits
      \cite{briend duval}.
Recall from the introduction that for a holomorphic map $f$ on $\pk$, 
we denote by $J^*$   the support of $\mu_f$.

 Repelling periodic points are dense in $J^*$. In general the closure of repelling periodic orbits can be strictly larger than $J^*$, which is a 
 source of technicalities in $J^*$-stability theory. Note however that  this phenomenon does not happen 
for regular polynomial skew products of $\cd$ (see \cite{jonsson}), which are basic to  all constructions in this paper. 

\medskip

Let now $(f_\la)_{\la\in \La}$ be a holomorphic family of holomorphic maps of degree $d$ on $\pk$ (or equivalently, a holomorphic map $\La\cv\hol_d(\pk)$). 
A  notion of stability for such a family was recently introduced in \cite{bbd}.  The relevance of this notion is justified by  its 
  number of natural  equivalent definitions.  
  The following is a combination of Theorems 1.1 and 1.6 in \cite{bbd}, with a slightly different terminology (see below for the notion of Misiurewicz bifurcation).

\begin{thm}[Berteloot, Bianchi \& Dupont]\label{thm:bbd}	
Let $(f_\la)_{\la\in \La}$ be a holomorphic family of holomorphic maps of degree $d$ on $\pk$. Then the following assertions are equivalent:
\begin{enumerate}
\item[(i)] The function on $\La$ defined by the 
sum of Lyapunov exponents of $\mu_{f_\la}$: $\la \mapsto \chi_1(\la)+ \cdots + \chi_k(\la)$ is  pluriharmonic   on $\La$. 
\item[(ii)] The sets $(J^*(f_\la))_{\la\in \La}$ move holomorphically in a weak sense. 
\item[(iii)] There is no (classical) Misiurewicz bifurcation in $\La$. 
\end{enumerate}
If in addition $\La$ is 
 a simply connected 
  open subset of $\hold(\pk)$ or if  $\La$ is any simply connected complex manifold and 
$k=2$, these conditions are equivalent to:
\begin{enumerate}
\item[(iv)] Repelling periodic points contained in $J^*$ move holomorphically over $\La$. 
\end{enumerate}
\end{thm}

If these  equivalent conditions  
 hold we say that the family is {\em $J^*$-stable} over $\La$.  For an arbitrary family we thus have a maximal open set of stability, the 
{\em stability locus} and its complement is by definition the {\em bifurcation locus}. We denote by $\bif$ the bifurcation locus. 

Condition {\em (ii)} will not be used below so  we don't need to explain what the ``weak sense'' in {\em (ii)} exactly is.  
By  ``classical Misiurewicz bifurcation" in {\em (iii)} we mean a proper intersection in $\La\times \pk$ 
 between a post-critical component of $\widehat f$ and the graph of a holomorphically moving repelling periodic point $\gamma:\om\cv\pk$ over 
  some open set $\om\subset \La$. We will generalize   this notion in detail  in \S\ref{subs:repeller} below (see in particular 
  Definition \ref{def:proper} for the notion of proper intersection). 
 
   \medskip

Let us also quote the following result, which  is contained in Proposition 2.3 and Theorem 3.5 in \cite{bbd}. 

\begin{prop}\label{prop:briend duval}
Let $(f_\la)_{\la \in \La}$ be a holomorphic family of holomorphic mappings of degree $d$ on $\pk$, with $\La$ simply connected. 
Assume that there exists a holomorphic map $\gamma: \La\cv\pk$ such that for every $\la\in\La$, $\gamma(\la)$
 is disjoint from the post-critical set of $f_\la$. Then $(f_\la)$ is $J^*$-stable. 
  \end{prop}

  \subsection{Hyperbolic repellers in $J^*$}\label{subs:repeller}
  Recall that if $E_0$ is an invariant hyperbolic compact set for $f_0$, then for $f\in \hold$ sufficiently close to $f_0$, there exists a hyperbolic set 
  $E(f)$ for $f$ and a conjugating homeomorphism $h = h(f):E_0 \cv E(f)$. The conjugacy can be chosen to depend continuously on $f$, with
  $h(f_0) = \mathrm{id}$, in which case it is unique  and is actually a holomorphic motion. We will refer to it as the {\em continuation} of $E_0$. 
  
According to the usual terminology, a {\em basic repeller} $E$ is a locally maximal and transitive hyperbolic repelling invariant set
 for $f_0$.  It is classical that if $E$ is a basic repeller, then repelling periodic points are  dense    in $E$. 
 Notice also that for an open map, the local maximality property
  for a hyperbolic repelling set is automatic, see \cite[\S 6.1]{prz urb}. Since holomorphic maps on $\pk$ are finite-to-one, they are open so this applies in our context. 
 
     \begin{lem}[see also   \cite{bianchi},  Lemma 2.2.15]\label{lem:continuation}
 Let $f_0$ be a holomorphic map on $\pk$ of degree $d\geq 2$. Let  $E_0$ be a basic repeller for $f_0$ 
 such that $E_0\subset J^*(f_0)$.  
 Then there exists a neighborhood $\om$ of $f_0$ in $\hold$ such that  for $f\in \om$ the continuation 
   $E(f)$ of $E_0$ is well-defined and is contained in $J^*(f)$. 
   \end{lem}
 
 \begin{proof}
 By assumption there exists a neighborhood $N$ of $E_0$ such that $\bigcap_{k\geq 0} f^{-k}(N) = E_0$, where we restrict to the preimages remaining in $N$. 
We can assume that $N$ is a $r$-neighborhood of $E_0$ for some $r= 2r_0>0$. 
Since $E_0$ is contained in $J^*(f_0)$, for every $x\in E_0$ we have that  $\mu_{f_0}(B(x, r_0))>0$. Furthermore 
by compactness  
there exists  $\delta_0$ such that 
for every $x\in E_0$,   $\mu_f(B(x,r_0/2))\geq \delta_0$.

Now, for $f$ close enough to $f_0$, $N = \bigcup_{x\in E_0} B(x, 2r_0)$ is also an isolating neighborhood for $E(f)$.  
In addition, the
uniform continuity of the potential of the maximal entropy measure implies that if $x\in E_0$, 
$\mu_{f}(B(x, r_0/2)) \geq \delta_0/2$. 
By construction, if $f$ is close enough to $f_0$, 
every  $y$ in $E_f$ is the intersection of a sequence  of nested preimages: $\set{y}  = \bigcap_{k\geq 0} f_x^{-n}(B(f^n(y), r_0))$, 
where $f_y^{-n}$ denotes the inverse branch of $f^n$ mapping $f^n(y)$ to $y$, and  
$f^n(y)$ is $r_0/2$ close to $E_0$. Therefore  $B(f^n(y), r_0) \supset B(x, r_0/2)$ for some $x\in E_0$, hence 
$\mu_f ( B(f^n(y), r_0))\geq \delta_0/2$, and finally $y\in \supp(\mu_f)$. 
\end{proof}

We now introduce a notion of proper intersection between a hyperbolic set and the post-critical set. 

\begin{defi}\label{def:proper}
Let $(f_\la)_{\la\in \La}$ be a holomorphic family of holomorphic maps on $\pk$ of degree $d\geq 2$ on $\pk$ and let $\lo\in \La$.
 Let $E_0$ be a basic  repeller for $f_0$. 
We say that   a post-critical component  $V_0\subset f_0^n(\crit(f_0))$ intersects $E_0$ properly if  there exists $x_0 \in V_0\cap E_0$ 
such that the graph $\widehat x$ of the   continuation of  $x_0$ as a point of $E_0$ is not contained in $\widehat V$, where $\widehat V$ is the irreducible component of $\widehat f^n (\crit(\widehat f))$ containing $V_0$. 
\end{defi}

The terminology is justified by the fact  that since $\widehat f^n (\crit(\widehat f))$ is of codimension 1 in $\La\times \pk$, 
the assumption on  $\widehat x$ implies that the graph   $\widehat x$ and $\widehat f^n (\crit(\widehat f))$ intersect properly at $(\lo, x_0)$.
Note that this terminology is slightly inadequate since the notion depends on the family and not only on $f_0$.  

 The mechanism leading to robust bifurcations will be based on the following lemma.

   \begin{propdef}\label{propdef:basic}
   Let $(f_\la)_{\la\in \La}$ be a holomorphic family of holomorphic maps on $\pk$ of degree $d\geq 2$ on $\pk$. 
  Assume that for some $\lo\in \La$, there exists a basic repeller
 $E_0$, contained in $J^*(f_0)$ and an integer 
  $k\geq 1$ such that  $f_0^n(\crit(f_0))$ intersects $E_0$ properly. 
  
  Then $\lo$ belongs to the bifurcation locus of the family. If such a situation happens, we  say that a Misiurewicz bifurcation occurs at $\lo$
  \end{propdef}

\begin{proof}
Fix a sequence of repelling periodic points $x_j$ in $E_0$,  converging to $x_0$. 
The stability of proper intersections in analytic geometry \cite[\S12.3]{chirka} implies that   there exists a sequence 
$\la_j\cv \lo$ such that 
 $\widehat x_j(\la_j)$ properly intersects  
  $\widehat E(\la_j)$. Then by  Theorem \ref{thm:bbd} we infer that $\la_j$ belongs to $\bif$, and so does $\lo$.  
\end{proof}

It follows  that in condition {\em (iii)} of Theorem \ref{thm:bbd} we can replace classical Misiurewicz bifurcations by Misiurewicz bifurcations in this sense.

  \section{Robust bifurcations from topology}\label{sec:topology}
  
  As outlined above, to construct robust bifurcations, we need to find situations in which  the post-critical set has a robust intersection with a 
  basic repeller.
  Our first method to produce such a robust intersection is based on topological arguments.
  
The  starting parameter is a product map in $\cd$
of the form  $f(z,w) = (p(z), w^d)$, where $p$ is of degree $d$. Then $f$ extends to $\pd$ as a holomorphic map. The small Julia set of $f$ is $J^*(f) = J_p\times S^1$, where $S^1$ denotes the unit circle in $\cc$.

  \begin{thm}\label{thm:topology}
  Let $f(z,w) = (p(z), w^d)$ be a product map with $\deg(p) = d$. Assume that $p$ satisfies the following properties:
  \begin{itemize}
  \item[(A1)] there exists a basic hyperbolic repeller   $E$ for $p$ that disconnects the plane;
  \item[(A2)] $p$ belongs to the bifurcation locus in $\poly_d(\cc)$. 
  \end{itemize}
  Then  $f$ belongs to the closure of the interior of the bifurcation locus in $\hold(\pd)$. More precisely there exists
  a polynomial $q$ of degree $\leq 2$ 
  and a sequence $\e_j\cv 0$  such that  
  the map $f_\e\in \hold(\pd)$ defined by 
  $f_\e(z,w) = (p(z) + \e q(w), w^d)$ belongs to $\mathring{\bif}$ for $\e=\e_j$.
  \end{thm}
 
 Notice that the assumption on $p$ requires\footnote{Indeed, we essentially   need both an attracting orbit and an active critical point.} $d\geq 3$, 
 so   the perturbation $\e h(w)$ does not affect the highest
degree  part of $f$, hence $f_\e\in \hold(\pd)$. 

  We denote by 
 $\poly_d(\cc)$ the space  of polynomials of degree $d$ in $\cc$.
Examples of polynomials $p\in \poly_d(\cc)$ satisfying the assumptions of Theorem \ref{thm:topology} are abundant. 
One may be tempted to think that as soon as $p$ admits one active critical point and an attracting periodic orbit,
 then there is a nearby $\widetilde p$ satisfying the assumptions of the theorem. 
The next two corollaries are results in this direction. They will be proven at the end of this section. 
 
\begin{cor}\label{cor:sinks}
 Let $f(z,w) = (p(z), w^d)$, with  $\deg(p) = d\geq 3$. Assume that $p$ is a bifurcating polynomial in $\poly_d(\cc)$, such that $(d-2)$ of its 
  critical points are attracted by periodic sinks. Then $f$ belongs    to the closure of the interior of the bifurcation locus in $\hold(\pd)$.
\end{cor}

When several critical points are active it is convenient to use the formalism of bifurcation currents 
  (see \cite{survey} for an introduction to this topic; the bifurcation current is denoted by $\tbif$).

\begin{cor}\label{cor:tbif}
 Let $f(z,w) = (p(z), w^d)$, with  $\deg(p) = d\geq 3$. Assume that in $\poly_d(\cc)$, 
 $p\in \supp(\tbif^k)$  for some $1\leq k\leq d-1$
  and that   $(d-1-k)$ critical points are attracted by sinks.
Then $f$ belongs    to the closure of the interior of the bifurcation locus in $\hold(\pd)$.
\end{cor}

Notice that this holds in particular when $p$ belongs to the Shilov boundary of the connectedness locus (corresponding to the case 
$k=d-1$).

   \begin{proof}[Proof of Theorem \ref{thm:topology}]
Observe first that it is enough to prove the result under the assumption 

\smallskip

 {(A2')}  {\it there exists a  simple   critical point $c$ for $p$ and an integer $k\geq 1$  such that $p^k(c)\in E$.}
 
 \smallskip 
 
 Indeed  let $p_0$ be a polynomial satisfying (A1) and (A2). Since $p_0$ belongs to the bifurcation locus, it has an active critical point $c$. Taking a branched cover of 
 parameter space, we may always assume that $c$ can be followed holomorphically as $p\mapsto c_p$.  
 Also, $E$ locally persists as a repelling set $E_p$ in a neighborhood of $p_0$.
 A well-known and elementary normal families argument shows that  
    there exists an arbitrary small perturbation     $p_1$  of $p_0$
     and an integer $k$ 
    such that $  p_1^k(c_{p_1})\in E_{p_1}$. 
    Furthermore   the set of polynomials  $p\in \poly_d$ possessing  
a multiple critical point is algebraic. Since $E$ is infinite, the set of polynomials  $p\in \poly_d$ such that $p(c_p)\in E$ is not locally analytic near $p_1$. 

Altogether, it follows that there is a sequence of polynomials $p_j$ satisfying (A1) and (A2')
 and converging to $p_0$. Thus, if the theorem has been shown to  hold under the assumptions (A1) and (A2')
  we simply pick such a sequence 
 and get that the result holds for $p_0$ as well.

\medskip

In a first stage let us prove the theorem in the case where $k=1$. 

\medskip
  
  \noindent{\bf Step 1:} Topological stability  of the hyperbolic set.
  
  For $\e=0$, $f=f_0$ admits a basic repelling set $\E = E\times S^1$. Notice that $E$ has empty interior in the plane. 
  If $g\in \hold$ is sufficiently close to $f$,    
  $\E = \E({f})$   admits a continuation $\mathcal{E}({g})$ as a hyperbolic set, that is, there exists  a continuous (even holomorphic in $g$)
       family of equivariant homeomorphisms  $h_{f,g}: \E(f)\cv \E({g})$.
   Actually there is more: 
   
   \begin{lem}\label{lem:extension}
   The conjugating  homeomorphism $h_{f, g}$ can be extended to a homeomorphism of $\pd$ (not compatible with the dynamics)
   depending continuously on $g$ and such that $h_{f, f} = \mathrm{id}$.
   \end{lem}

   \begin{proof}
   The result  is more or less part of the folklore.   
 For completeness we sketch the argument. 
    Let $U$ be an isolating neighborhood of $\Lambda$, which can be assumed to have smooth boundary, 
    and $\mathcal N$ be  a neighborhood of $f$ in $\hold$ such that for $g\in \mathcal N$, 
  $\bigcap_{n\geq 0} g^{-n}(U) = \E(g)$ (here and in the rest of the proof only 
    preimages remaining in $U$ are taken into account). By the hyperbolicity 
    assumption (and reducing $U$ and $\mathcal N$ a bit if necessary), there exists $\delta>0 $ and $\lambda>1$ 
    so that if $x,y\in U$ are such that $f^k(x)$ and  $f^k(y)$ remain in  $U$ and 
    $d(f^k(x), f^k(y))\leq \delta$ for $0\leq k\leq n$ then $d(x,y)<\delta \lambda^{-n}$, and the same is true for every $g\in \mathcal N$.  Reducing $U$ and $\mathcal N$  again  if necessary, we may assume that for every $g\in \mathcal N$ and $x\in U$
    different preimages of $x$ under $g$ stay far apart, so that 
     each such preimage can be followed unequivocally 
     as $g$ varies in $\mathcal N$. In addition we can assume that for every $x$, the continuation $g_{-1}(x)$ of any given preimage  $f_{-1}(x)$ of $x$
     satisfies 
     $d(f_{-1}(x), g_{-1}(x))<\e$ for  $g\in\mathcal N$,  where $\e< \delta (1-1/\la)/3$. 
     
   Given such   $g$ we will construct the conjugacy $h = h_{f,g}$ 
   by starting from  $\fr U$. Fix $h = \id$ on $\fr U$,  then define $h : f^{-1}(\fr U)\to g^{-1}(\fr U)$ by assigning to every $y\in f^{-1}(\fr U)$ the corresponding preimage of $f(y)$ under $g$, and  finally  extend $h$ to a homeomorphism  from $\overline U\setminus f^{-1}(U)$ to $\overline U\setminus g^{-1}(U)$. In addition we can ensure that $h$ is $\e$-close to the identity. 
    Now for every $x\in \overline U\setminus \E(f)$ we set $h(x) = g_{-k} \rond  h \rond  f^k(x)$, where $k$ is the last integer 
    such that $f^k(x) \in  U$, and $g_{-k}$ is the inverse branch of $f^k$ at $f^k(x)$ obtained by continuation of $f_{-k}: f^k(x)\mapsto  x$. 
    In this way we get an equivariant homeomorphism $h: \overline U\setminus \E(f)\to \overline U\setminus \E(g)$.
    
    \medskip
    
    The point is to show that this conjugacy extends continuously to $\E(f)$. Then the resulting extension will be a homeomorphism by reversing the roles of $f$ and $g$, which simply extends to $\pd$ by declaring that $h= \id$ outside $U$, thereby concluding the proof.     
    
    To prove that such an extension exists, 
    we need to prove that for every $x\in \E$, if  
    $(x_n) \in (U\setminus \E)^\nn$ is any sequence converging to $x$  
    then $(h(x_n))$ converges and its limit does not depend on $(x_n)$. 
    
    \medskip
    
    As a first step towards this result, let us show that 
    for every $x\in U\setminus \E$, $d(x,h(x))<\delta/3$. 
For this, let $k$ be the last integer such that $f^k(x)\in U$ and let us 
  show by   induction on $q$ that 
$d(f^{k-q}(x), g^{k-q}(h(x)))\leq \e\sum_{j=0}^{q} \la^{-j}$. Then the result follows from our choice of $\e$ by putting $q=k$. 
For $q = 0$, $g^k(h(x)) = h(f^k(x))$, and $d(f^k(x), h(f^k(x)))\leq \e$ by definition of $h$. Now assume the result has been proved for some $0\leq q\leq k-1$. Then 
 \begin{align*}
 d&( f^{k-q-1}(x), g^{k-q-1}(h(x)))\\ & \leq d (f_{-1}(f^{k-q}(x)), g_{-1}(g^{k-q}(h(x)))) \text{ where } f_{-1}, \ g_{-1} \text{ are appropriate inverse branches}\\
 &\leq  d (f_{-1}(f^{k-q}(x)), g_{-1}(f^{k-q}(x))) + d (g_{-1}(f^{k-q}(x)), g_{-1}(g^{k-q}(h(x))))\\
 &\leq \e + \unsur{\lambda }d ( f^{k-q}(x),  g^{k-q}(h(x)))
 \end{align*} 
where the last inequality follows from the definition of $\e$ and  the induction step. The estimate $d(x,h(x))<\delta/3$ follows. 

\medskip

We are now in position to conclude  the argument. If $(x_n)$ and $(y_n)$ are two sequences converging to $x$, then for large $n$, the orbit of 
$x_n$ (resp. $y_n$) shadows that of $x$ for a long time. More precisely there exists  $k = k(n)\underset {n\cv\infty}{\longrightarrow} \infty$ 
such that for $j\leq k$,
$d(f^j(x_n), f^j(y_n))< \delta/3$. Therefore $$d(h(f^j(x_n)), h(f^j(y_n))) = d(g^j(h(x_n)), g^j(h(y_n)))<\delta$$ and we conclude that 
$d(h(x_n), h(y_n))<\delta \la^{-k(n)}$. The same argument applied to $x_n$ and $x_m$ for large $n,m$ implies that 
$(h(x_n))$ is a Cauchy sequence. Altogether, we infer that the sequence $(h(x_n))$ converges to a limit which depends only on $x$. This implies 
that the conjugacy extends to $\E$ and finishes the proof of the lemma.  
   \end{proof}

 \noindent{\bf Step 2:} Analysis of $f_\e$.  
 
 In this paragraph we fix $f_\e$ as in the statement of the theorem (an explicit expression for $q$ will be given afterwards), and work only within the family 
 $(f_\e)$, where $\e$ ranges in a small neighborhood of $0\in \cc$. 
 For notational ease, all dynamical objects will be indexed by $\e$ (like $\mathcal E_\e$, etc.). The problem is semi-local around $\E$ so we work in $\cd$. 
 Also, we normalize the $z$-coordinate so that  $p(c) = 0$.  

For every $\e$, the vertical line $C = \set{c}\times \cc$ is a critical component of $f_\e$. We set  $V_\e = f_\e(\set{c}\times \cc)$, which is parameterized by
\begin{equation}\label{eq:Veps}
V_\e = \set{(\e q(w), w^{d}), w\in \cc}.
\end{equation} The solid torus $\cc\times S^1$ is totally invariant under the dynamics and contains $\E_\e$. 
For small enough $\e$, the hyperbolic set $\E_\e$ is homeomorphic to and close to  $E\times S^1$. 
Then from (A1) and  Lemma \ref{lem:extension}
(applied to $f\rest{\cc\times S^1}$)
 its complement in $\cc\times S^1$ admits both bounded connected components (the {\em inside} part of $\E_\e$) and a unique 
unbounded one (the {\em outside}). 

The following lemma provides a form of ``topological transversality" between $V_\e$ and $\mathcal{E}_\e$. 

\begin{lem}\label{lem:epsj}
There exists a polynomial $q$ of degree $\leq 2$ and a sequence 
$\e_j \cv 0$
such that for   $\e = \e_j$, 
 the real analytic curve 
$$v_\e  = V_\e \cap (\cc\times S^1) = \set{(\e q(w), w^{d}), \abs{w} = 1}$$ 
intersects 
both the inside and the outside  of $\E_\e$. 
\end{lem}


\begin{proof}
 Take $q$ of the form $q(w)  =  (w-1) \tilde q (w)$. Then the fiber $\set{w=1}$ is invariant under $f_\e$.  Since in addition 
$f_\e(z,1 ) = (p(z),1)$ does not depend on $\e$, we get that 
$\E_\e \cap \set{w=1} = E\times \set{1}$ (indeed, if  $x\in E\times \set{1}$ the $f_\e$-orbit of $x$ trivially 
shadows its $f$-orbit). The curve $v_\e$ intersects  $\set{w=1}$ in $d$ points 
$\e q(\zeta^k)$, where $\zeta = e^{\frac{2i\pi}{d}}$ (here we identify  $\set{w=1}$ with $ \cc$ and  
 drop the second coordinate). Notice that for $k=0$, $\e q(1) = 0 $ belongs to  
$E$.  Put $\zeta^{\pm} = \zeta^{\pm 1}$. 

Denote by $\mathrm{Out}(E)$ (resp. $\mathrm{Inn}(E)$)  the unbounded component of (the union of bounded connected components of) 
the complement of $ E$ in $\cc$, 
which by assumption are both non-empty. Remark that the uniform hyperbolicity and transitivity of $E$ imply that every $x\in E$ is 
accumulated both by  $\mathrm{Out}(E)$  and  $\mathrm{Inn}(E)$.

\medskip

\noindent{\bf Claim:} there exists $\alpha$ so that for $q(w) = (w-1)(w-\alpha)$, there exist $\e_j \cv 0$ so that 
$\e q(\zeta^-)$ (resp.  $\e q(\zeta^+)$)  belongs to $\mathrm{Inn}(E)$ (resp. $\mathrm{Out}(E)$).

\medskip


Indeed let    $\alpha\in \re$ be such that $ \frac{\zeta^+-\alpha}{\zeta^--\alpha} = e^{i\theta}$ with $\theta\notin \pi \mathbb Q$ and 
assume  that the  claim is false. 
Then for every  small enough $\e$, 
 $\e q(\zeta^-) \in \mathrm{Out}(E)$ implies that  $\e q(\zeta^+) \notin \mathrm{Inn}(E)$. 
 Using the above remark we reformulate this as
 $$\e q(\zeta^-) \in \mathrm{Out}(E) \Rightarrow \e q(\zeta^+) \in  \cc\setminus   \mathrm{Inn}(E) =  \mathrm{Out}(E)\cup E = \overline {\mathrm{Out}(E)}$$
 so by continuity we conclude that 
 $$\e q(\zeta^-) \in \overline { \mathrm{Out}(E)} \Rightarrow\e q(\zeta^+) \in \overline { \mathrm{Out}(E)}.$$ 
 Of course the reverse implication holds by symmetry, and the same assertion holds for $\mathrm{Inn}(E)$ by taking the complement. 
Finally we conclude that   
   for every small enough $\delta \in \cc$, 
\begin{equation}\label{eq:rotation}
\delta \in   { \mathrm{Inn}(E)}  \Leftrightarrow  \delta \frac{q(\zeta^+)}{q(\zeta^-)}\in  { \mathrm{Inn}(E)} .
\end{equation}
Observe that 
  $$\frac{q(\zeta^+)}{q(\zeta^-)} = \frac{\zeta^+-1}{\zeta^--1} \frac{\zeta^+-\alpha}{\zeta^--\alpha} =  - e^{\frac{2i\pi}{d}} \frac{\zeta^+-\alpha}{\zeta^--\alpha} = - e^{\frac{2i\pi}{d}}  \e^{i\theta} =  \e^{i\theta'}, \text{ with }  {\theta'} \notin \pi \mathbb Q, $$
 therefore 
 by \eqref{eq:rotation},
 $\mathrm{Inn}(E)$  is invariant under an irrational rotation about 0. 
Now if $\om$ is an open ball in   $\mathrm{Inn}(E)$ close to 0, its orbit under this irrational rotation contains a circle about 0, which contradicts the fact that $0\in \overline{ \mathrm{Out}(E)}$, and concludes the proof of the claim.

     Note that when 
      $E$ is a Jordan curve, which is often the case in practice (see Lemma \ref{lem:active}),
      the argument can be simplified a bit. In particular 
     $E$ cannot be invariant under multiplication by $e^{\frac{2i\pi}{d}}$ at 0 so
      we can simply choose $q(w) = w-1$. 
      
      \medskip
 
 Summing up,  we have shown that for $\e = \e_j$, $v_\e\cap \set{w=1}$ intersects different connected components of $\set{w=1}\setminus E$.    
Thus to conclude the proof of the lemma  it is enough to prove 
that for   $\e$ small enough the map $x\mapsto \mathrm{Component}(x)$ induces a     
 1-1 correspondence between the
  connected components of $\set{w=1}\setminus E$ 
and the  connected components   of $(\cc\times S^1)\setminus  \mathcal E_\e$. 

Indeed this is obvious for $\e=0$. Now when we vary $\e$, 
Lemma \ref{lem:extension} shows that there is a homeomorphism $h_\e$ of $\pd$ such that $h_\e(\mathcal E) = \E_\e$, which is homotopic to 
the identity. Since $f_\e\rest{\set{w=1}}$ does not depend on $\e$, 
the construction of $h_\e$ shows that 
we can choose $h_\e = \mathrm{id}$ on this fiber. Thus for every $x\in \set{w=1}\setminus E$, $h_\e$ sends the connected 
component of $ (\cc\times S^1)\setminus \mathcal E$ containing $x$ to 
the connected component of $ (\cc\times S^1)\setminus \mathcal E_\e$ containing $x$, which was the desired claim. 
This finishes the proof of Lemma \ref{lem:epsj}. 
\end{proof}
  
 \noindent{\bf Step 3.} Robustness of the Misiurewicz phenomenon in $\hold(\pd)$.
 
Consider an arbitrary holomorphic map $f'\in \hold(\pd)$ close to $f_0$.
 Let us first analyse how the post-critical component $V_0$ can be continued for $f'$  (this has been done explicitly before for the family $(f_\e)$).
 Again, the problem is semi-local around $\E$ so for convenience we work 
 in coordinates in $\cd$.
 The vertical line    $\set{c}\times \cc$ is a component of multiplicity 1 of $\crit(f_0)$, 
 and moreover $\crit(f_0)$ is smooth along $\set{c}\times \cc$  outside 
 $\set{w=0}$. 
 Therefore every compact piece of $(\set{c}\times \cc)\setminus \set{w=0}$ can be followed as a part of the critical set for $f'$ close to $f_0$. 
 More specifically, if $r$ is so small that $D(c,r)$ contains no other critical point of $p$, and if 
 we  let $U = D(c, r)\times \set{1/10 <  \abs{w} < 10}$, there exists a neighborhood $N(f_0)$ such that for every 
 $f'\in N(f_0)$, $\crit(f')\cap U$ is of multiplicity 1 and consists in an annulus $A(f')$ 
 which is a graph over $A = \set{c}\times  \set{1/10 <  \abs{w} < 10}$, and  converges to $A$ as $f'\cv f_0$. 
 Finally we define  $V(f') $ to be  the component of $f'(\crit(f'))$ containing  $ f'(A(f'))$. 
 
 \medskip

 Fix a continuous function 
  $\varphi_0: \cc\times S^1\to \re$ (depending only on $z\in \cc$) such that $\E_0 = \set{\varphi_0= 0}$, $\mathrm{Inn}(\E_0)  = \set{\varphi_0<0 }$ and 
 $\mathrm{Out}(\E_0)  = \set{\varphi_0>0 }$. Extend it to $\cd$ by putting $\varphi_0(z,w):=\varphi_0(z)$. 
 Define  also $\psi_0$   by $\psi_0(z,w) = \abs{w}-1$, and let $\Phi_0=(\varphi_0, \psi_0): \cd\cv \re^2$, so that 
 $\E_0 = \set{\Phi_0 = 0}$. By Lemma \ref{lem:extension}, for $f'$ close to $f_0$, there exists a continuous map
 $\Phi_{f'} = (\varphi_{f'}, \psi_{f'}):\cd\cv\re^2$ such that $\E({f'}) = \set{ \Phi_{f'} =0}$. Notice that for $f' = f_\e$,
  since $\E_\e\subset \cc\times S^1$, 
 we can choose $\psi_{f_\e}=:\psi_\e =\psi_0$. 
 
 Let now $f' = f_{\e_j}$ be as in  Lemma \ref{lem:epsj}. 
  We will use elementary algebraic topology to show that $V(f')$ intersects $\E({f'})$ for $f'$ close to $f_{\e_j}$. 
 Working in the natural parameterization of $V_\e$ given in \eqref{eq:Veps}, 
the real valued   function $\varphi_{\e_j}$ changes sign along the 
 segment $[\zeta^-, \zeta^+]$  of the real curve $v_{\e_j}$, say $\varphi_{\e_j}(\zeta^-) < 0 < \varphi_{\e_j}(\zeta^+)$. 
 In other words, $\varphi_{\e_j}\circ f_{\e_j} (c, \cdot)$ changes sign along the 
 segment $[\zeta^-, \zeta^+]$  of the unit circle. 
 
Consider a simple loop $\ell$ in the complex plane, starting at $\zeta^-$, then joining it to $\zeta^+$ outside the unit disk, 
 and then returning back to    $\zeta^-$ inside the unit disk, and  staying  in the annulus $A$. 
 (e.g. we can take the exponential of the oriented 
  boundary of the rectangle $[1-\rho, 1+\rho]\times \left[-\frac{2\pi}{d},\frac{2\pi}{d}\right]$ for some small $\rho>0$.)  
 Then  the loop $f_{\e_j}(\set{c}\times {\ell})$  is disjoint from $\E_{\e_j}$ and 
 the winding number of  $\Phi_{\e_j}$ along $  f_{\e_j} (\set{c}\times\ell)$ is equal to 1. 
 
 For  $f'$ sufficiently close to $f_{\e_j}$, we can lift $\set{c}\times\ell$ to a loop $\ell'$ contained 
 in $A(f') \subset \crit(f')$. By continuity,
  the winding number of $\Phi_{f'}$ along $f'\circ \ell'$ is 1.  Since $\ell'$ bounds a disk in $A(f')$, we infer that 
 $\Phi_{f'}$ must vanish  on $f'(A(f'))$. In other words, $V(f')$ intersects $\E({f'})$, which was the desired result. 
 
 To conclude that a robust Misiurewicz bifurcation occurs at $f_{\e_j}$, it remains to check that   the intersection between  $V(f')$ and 
  $\E(f')$ that we have produced is proper. For this, it is enough to observe that there are product maps $(p_1(z), w^d)$, with $p_1$ arbitrary close to $p$,   
 such that if $c_1 \in \crit(p_1)$ is the critical point   continuing $c$, then $p_1(c_1)\notin E(p_1)$. Indeed $c$ is active at $p$ in $\poly_d(\cc)$ so under arbitrary small perturbations, it can be sent into the basin of a sink. This finishes the proof of the theorem in the case where $k=1$. 
 
  \medskip  
   
  It remains to treat the case where $k$ is arbitrary. 
  Observe that we can assume that $p^j(c)$ is not critical for $1\leq j\leq k$, otherwise we replace 
  $c$ by the last appearing critical point in this orbit segment. The structure of the proof is the same. Step 1 doesn't need to be modified.
  
 For Step 2, since  in the case $k=1$ we have used 
   the explicit parameterization of $f_\e(\set{c}\times \cc)$,   here
    rather than working with $f_\e^k(\crit(f_\e))$, 
   we show that $f_\e (\crit(f_\e))$   intersects $f_\e^{-(k-1)}(\E_\e)$ in topologically transverse manner
    for a well-chosen sequence $(\e_j)$.  The only slight difficulty to keep in mind when pulling back 
    is that $\E_\e$ can intersect other components of the post-critical set. 
    
  We put     $F = p^{-(k-1)}(E)$ and normalize the $z$ coordinate so that $p(c) = 0 \in F$.
As in the previous case let $q$ be of the form $q(w) = (w-1)(w-\alpha)$.
      Let 
    $\F_\e = f_\e^{-(k-1)}(\E_\e)$ which is contained in $\cc\times S^1$. 
    For $\e =0$, $\F_0 = F\times S^1$ but now $\F_\e$ needn't be homeomorphic to   $\F_0$. 
     Since $f$ is proper and dominant $\F_\e$  is a compact set  with empty interior. 
     Thus if $\om$ is a bounded connected component of $(\cc\times S^1)\setminus \E_\e$, 
     $f_\e^{-(k-1)}(\om) $ is a union of bounded components of $(\cc\times S^1)\setminus \F_\e$.  It follows that in $\cc\times S^1$, 
    $\mathrm{Inn} (\F_\e)  = f_\e^{-(k-1)}(\mathrm{Inn} (\E_\e))$ and likewise for $\mathrm{Out} (\F_\e)$. 
    
    We   have that $\F_\e\cap \set{w=1} = F$, and if $(z,1)\in \mathrm{Inn}(F)$ in $\cc$, then  $(z,1)\in \mathrm{Inn}(\F_\e)$ in $\cc\times S^1$, and likewise for the outer part. Indeed  $f^{k-1}(z,1) = (p^{k-1}(z),1)$ and $p^{k-1}(z) \in \mathrm{Inn}(E)$ so $(p^{k-1}(z),1)  \in \mathrm{Inn}(\E_\e)$ and the result follows by pulling back the corresponding inner component. So we can argue exactly as in Lemma \ref{lem:epsj} to conclude that there exists a 
    sequence $\e_j\cv 0$ such that the real curve $v_e = f_\e(\set{c}\times \cc )\cap( \cc\times S^1)$ intersects both 
  $\mathrm{Inn}(\F_\e)$ and $\mathrm{Out}(\F_\e)$.

  This analysis being done, we can push forward by $f_{\e_j}^{k-1}$ to get a segment of the real analytic curve 
  $f_{\e_j}^k (\set{c}\times \cc)\cap (\cc\times S^1)$ whose endpoints belong to $\mathrm{Inn}(\E_{\e_j})$ and 
  $\mathrm{Out}(\E_{\e_j})$ respectively and apply the argument of Step 3 without modification. The proof is complete. 
    \end{proof}
  
  Corollary \ref{cor:sinks} is a direct consequence of the following lemma. 
  
\begin{lem}\label{lem:active}
Let $p$ be a polynomial of degree $d\geq 3$ with marked critical points $c_1, \ldots , c_{d-1}$. Assume that $c_2, \ldots c_{d-1}$ are attracted by attracting cycles and that $c_1$ is active. Then there exists $\widetilde p$ arbitrary close to $p$ satisfying  assumptions (A1) and (A2).
\end{lem}

\begin{proof}
Observe first that if $p$ is an arbitrary polynomial and 
$\mathcal{B}$ is the immediate basin of attraction of an attracting cycle of period $\ell$, such that 
$\fr\mathcal B$ contains no critical point nor parabolic periodic points, then by  Mañé's lemma \cite{mane fatou}
  $\fr \mathcal{B}$ is a hyperbolic set, which obviously disconnects the plane. We claim that it is transitive (recall that local maximality is automatic in this case). Indeed
  letting $\om$ be a component of $\mathcal{B}$ ($f^\ell(\om) = \om$), it is enough to show that $f^\ell\rest{\fr \om}$ is transitive.   
   It  follows from uniform hyperbolicity that  $\fr \om$ is locally connected, therefore the Riemann map $\phi: \dd\cv\om$ extends by continuity to
   $\phi: \overline\dd\cv\overline\om$. In addition 
    $\phi\rest\dd$ must be injective otherwise contradicting the polynomial convexity of $\om$. Thus   $\fr\mathcal{B}$ is a Jordan curve. 
    (This result actually holds  in much greater generality, see \cite{roesch yin}).   

The formula $\phi f^\ell \phi^{-1}$  defines a self-map of the  unit disk,  which can be extended by Schwarz reflexion to a rational map on $\pu$ preserving the unit circle, hence a Blaschke product. Now the Julia set of a Blaschke product is either the unit circle or a Cantor set inside the unit circle, and the second case occurs only if there is an attracting or parabolic fixed point on the circle, which is not the case here. We conclude that the Julia set is the circle, hence topologically transitive (even topologically mixing). Coming back to $f$ we 
conclude that $f^\ell\rest{\fr \om}$ is topologically mixing, which yields the desired claim.

 \medskip

Under the assumptions of the lemma, let 
 $\mathcal{B}$ be   the immediate attracting basin of the   cycle attracting, say, $c_2$, which is robust under perturbations, 
   and   consider the active  critical point $c_1$. 
   Since $p$ is a bifurcating polynomial, one can make a cycle of period $\ell m$ bifurcate close to $p$, so there exists 
$p_1$ close to $p$ with a parabolic point $\alpha$ of period $\ell m$ and multiplier different from $\pm 1$. Then:
\begin{enumerate}
\item $\alpha\notin \fr\mathcal B$: indeed $p_1^{\ell m}$ behaves like a rational rotation of order greater than $2$ at $\alpha$,  
 so there cannot exist an invariant  Jordan curve through $\alpha$. On the other hand by \cite{roesch yin}
 $\fr\mathcal{B}$ is a Jordan curve  invariant under $p^\ell$.
\item $\alpha$ must attract $c_1$ under iteration of $p_1^{\ell m}$, therefore $c_1\notin \fr\mathcal{B}$. 
\end{enumerate}
This implies that for this polynomial $p_1$, $\fr \mathcal B$ is a hyperbolic set  disconnecting the plane so (A1) holds. 
Since in addition  $p_1$ has a parabolic cycle, (A2) holds and the corollary follows. 
\end{proof} 

     \begin{proof}[Proof of Corollary \ref{cor:tbif}] 
We freely use the formalism of bifurcation currents, see \cite{survey} for an account. 
By passing to a branched cover, it is no loss of generality to assume that the critical points are marked as $c_1, \ldots, c_{d-1}$. Let $T_i$, $1\leq i \leq d-1$ be the associated bifurcation currents. Since $p\in \supp(\tbif^k)$, reordering the critical 
    points if necessary, we can assume that $p\in \supp(T_1\wedge \cdots \wedge T_k)$ and  that 
    $c_{k+1}, \ldots , c_{d-1}$ are attracted by sinks.  
    Then, using the continuity of the potentials of the $T_i$, we infer that 
    there exists a sequence of subvarieties  $[W_n]$ of codimension $k-1$ in $\poly_d$ along which $c_2, \ldots, c_k$ are periodic  
     and a sequence of integers $(d_n)$ 
    such that 
    $T_1\wedge \cdots \wedge T_k = \lim_{n\cv\infty} T_1\wedge d_n^{-1}[W_n]$ (see \cite[Thm 6.16]{preper}). 
    Therefore arbitrary close to $p$ there are polynomials with one active critical point and the remaining ones attracted by periodic sinks. The result then follows from Corollary \ref{cor:sinks}.     
     \end{proof}
     
  \section{Robust bifurcations from fractal geometry}\label{sec:blender}  

In this section we consider product mappings of the form $f(z,w) = (p(z), w^d+\kappa)$, where 
$p$ admits  a repelling fixed point $z_0$ that is almost parabolic, i.e. $1< \abs{p'(z_0)} < 1.01$ 
and $\kappa$ is large. 
We show that   certain  perturbations  of $f$ exhibit {\em blenders}, in the sense of Bonatti and Diaz \cite{bonatti diaz}, after proper rescaling.  
The degree 2 case requires a slightly different treatment so we  handle    the cases $d\geq 3$ and $d=2$ separately. Also we strive for   a
uniformity in $\kappa$ which allows to cover the parabolic case as well (see Corollary \ref{cor:parabolic}).

     \subsection{Estimates for Iterated Functions Systems in $\cc$}\label{subs:IFS}
  
  It is well-known \cite{daroczy katai, hare sidorov} that the limit set of the  IFS generated by two affine maps of the form $\ell_\pm (z) = \mu z \pm 1$ 
has empty interior when $\mu$ is non real and $\abs{\mu}$  is sufficiently close to 1. Since we will need precise estimates as well as variants of this result, we give a detailed treatment. It will be convenient for us to work with mappings that are contractions  on the unit disk, so we write 
them in a slightly different form.
  
  Let $\el = (\ell_j)_{j = 1}^{d}$ be a family of holomorphic
  maps  $\ell_j: \dd\to  \ell_j(\dd)\Subset \dd$. Since the $\ell_j$ contract the Poincaré metric they define an IFS whose {\em limit set} is 
  $$E =  \bigcap_{n\geq 0}\bigcup_{\abs {J} = n} \ell_{j_n}\circ \cdots \ell_{j_1} (\dd)$$ 
  (we use the classical multi-index notation $J = (j_1, \ldots, j_n)$). We say that the IFS  $\el$  satisfies the {\em covering property} on an open set $\Delta\subset \dd$ if   $\bigcup_{j=1}^d \ell_j(\Delta)\supset \overline \Delta$.  It then follows that $\Delta\subset E$ and   
   this property obviously persists for any small ($C^1$) 
  perturbation of $(\ell_j)$.
  
  \medskip
  
 Our first model situation concerns an  IFS with $d\geq 3$ branches roughly directed by $d^{\rm th}$ roots of unity. 
 
 \begin{lem}\label{lem:IFS}
 Let  $d\geq 3$ and $\el = (\ell_j)_{j=1}^{d}$ be the IFS generated by the affine maps $$\ell_j: z\mapsto m z + \alpha_j (1-\abs{m})
  e^{\frac{2\pi i }{d}j},$$ where $0.98<\abs{m} <1$ and 
  $  \alpha_j $ is such that 
  $\frac35<\abs{\alpha_j}<1$ and $\abs{\arg{\alpha_j}}< \frac{\pi}{20}$ 
  Then $\el$ satisfies the covering property on the disk $D\lrpar{0, \frac{1}{10}}$.
 \end{lem}
 
  \begin{proof}
Since $\abs{\alpha}<1$ we have that $\ell_j(\dd)\Subset \dd$ so we are in the above situation. Let now $z$ in $\overline D\lrpar{0, 1/10}$, 
we have to prove  that 
$\ell_j^{-1}(z)\in  D\lrpar{0, 1/10}$ for some $j$. We will show that if $\arg(z)\in  [-\pi/3, \pi/3]$ then $j=d$ (or equivalently $j=0$) is 
convenient. Then the other cases follow, since the unit circle is covered by translates of $ [-\pi/3, \pi/3]$ by $d^{\rm th}$ roots of unity. 

We compute $\ell_0^{-1}(z ) = \frac1m \lrpar{z- \alpha_0(1-\abs{m})}$. We have to  show that for $z = \rho e^{i\theta}$, with 
$\rho \leq \unsur{10}$ and $\abs{\theta}\leq \frac{\pi}{3}$, then $\abs{\ell_0^{-1}(z )} < \frac1{10}$, or equivalently 
$
\abs{z- \alpha_0(1-\abs{m})} < \frac{\abs{m}}{10}$. Now 
\begin{align*}
\abs{z- \alpha_0(1-m)}^2  &=  \rho^2 + \abs{\alpha_0}^2 (1-\abs{m})^2 - 2 \Re\lrpar{z \overline{  \alpha_0}(1-\abs{m})}\\
&\leq \rho^2 + (1-\abs{m})^2 - 2\rho \abs{\alpha_0} \cos(\arg(z) - \arg(\alpha_0))(1-\abs{m})\\
&\leq \rho^2 + (1-\abs{m})^2 - \frac25\rho(1-\abs{m})   
\end{align*}
where in last line we use $\abs{\alpha_0}>\frac35$ and $\cos(\arg(z) - \arg(\alpha_0))\geq \frac13 $.
So we are left to showing that 
$$\rho^2 + (1-\abs{m})^2 - \frac25\rho(1-\abs{m})  < \frac{\abs{m}^2}{10^2}, \text{ for } \rho\leq \unsur{10} \text{ and } 0.98<\abs{m} <1. $$ For the sake of this computation, put $1-\abs{m} = \delta< \frac{1}{50}$. The previous equation rewrites as 
$$\rho^2 + \delta^2 - \frac25 \rho\delta < \frac{(1-\delta)^2}{10^2} \Leftrightarrow \rho \lrpar{\rho - \frac25 \delta} < 
\frac{(1-\delta)^2}{10^2} -\delta^2. $$ Since $\rho\leq1/10$, to prove the last inequality it is enough to show that 
$$ {\rho - \frac25 \delta} < \frac{(1-\delta)^2}{10 } - 10 \delta^2
\Leftrightarrow \rho - \frac{1}{10} < \frac{1}{5} \delta - \frac{99}{10} \delta^2 = \frac{1}{5}  \delta \lrpar{ 1 -\frac{99}{2} \delta},$$
and the latter is true because the left hand side is non-positive  while the right hand side is positive.
  \end{proof}

\begin{rmk} \label{rmk:rotation}
Notice that conjugating by a rotation 
implies that for every fixed $\theta$, 
 the same result holds when $\alpha_j (1-\abs{m})$ is replaced by $e^{i\theta}\alpha_j (1-\abs{m})$ in the expression of $\ell_j$. 
\end{rmk}

  In view of 2-dimensional applications, it is useful to point out a slight reformulation  of
   this result.   Let $A$ be the range of allowed values for $\alpha$ in Lemma \ref{lem:IFS}, i.e. 
   \begin{equation} \label{eq:A}
   A = \set{\alpha \in \cc, \ \frac35<\abs{\alpha}< 1, \ \abs{\arg(\alpha)} < \frac{\pi}{20}}. 
   \end{equation} 
  
   \begin{lem}\label{lem:IFS variant}
 Let  $d\geq 3$ and $\el = (\ell_j)_{j=1}^{d}$ be the IFS generated by the affine maps $$\ell_j: z\mapsto m z + \alpha_j (1-\abs{m})
  e^{\frac{2\pi i }{d}j}. $$ where $0.98<\abs{m} <1$. Fix $\eta>0$ and assume that for every  $j = 1, \ldots , d$, the disk 
  $D(\alpha_j, \eta)$ is contained in $A$. 
  
 Then for every $z_0\in D\lrpar{0, \unsur{10}}$, there exists $j$ such that 
 $$\ell_j^{-1}(  D(z_0, \eta (1-\abs{m})))\subset D\lrpar{0, \unsur{10}}.$$
 \end{lem}
  
\begin{proof}  
Notice that when $\eta$ tends to $0$ this is precisely
 the statement of Lemma \ref{lem:IFS}. Now for the general case, simply observe that for every $z$,
 when $\alpha$ varies in a disk of radius $\eta$,  $\ell_j^{-1}(z)$ ranges in a disk of radius $\eta\frac{1-\abs{m}}{m}$. 
For $\eta >0$, put 
   $A_{\eta} = \set{\alpha \in A,  D(\alpha, \eta))\subset A}$ (beware that this is the opposite of a $\eta$-neighborhood). 
If for some $z$, 
 $$\ell_j^{-1} (z)\in D\lrpar{0, \frac1{10}} \text{  for every }\alpha\in A $$ (this is what we have proved for $z$ in an angular sector of width $\frac{2\pi}{3}$
 and $j = 0$) then 
 $$\ell_j^{-1} (z)\in D\lrpar{0, \frac1{10} - \eta\frac{1-\abs{m}}{m}} \text{  for every }\alpha\in A_\eta .$$ Now since 
 $\ell_j\inv (D(z,r))  = D\lrpar{z, \frac{r}{m}}$, we infer that 
 for $\alpha\in A_\eta$, $\ell_j\inv (D(z,\eta(1-\abs{m}))) \subset D\lrpar{0, \frac1{10}}$, 
which was the result to be proved
\end{proof}

 For an IFS $\el$, we define $\el^n$ to the the IFS generated by $n$-fold compositions of the generators of $\el$. It has the same limit set as $\el$.  
   Here is an analogue of Lemma \ref{lem:IFS} for an IFS with two branches. 
   In this case   the multiplier needs to be chosen close to the imaginary axis. 
  
  \begin{lem}\label{lem:IFS2}
  Let $\el = {\ell_\pm}$ be the IFS generated by the affine maps 
  $$\ell_\pm:z\mapsto m z \pm \alpha (1-\abs{m}), $$
  where $0.99<\abs{m} < 1$, $\abs{\arg(m)- \frac{\pi}{2}}<\frac{\pi}{50}$ and $\alpha$ is a complex number with $0.9<\abs{\alpha}<1$.  
   Then $\el^2$ satisfies the covering property on the disk $D\lrpar{0, \frac{1}{10}}$.
  \end{lem}

  \begin{proof}
To establish the result,  we show that 
$\el^2 = (\ell_+^2, \ell_+\circ \ell_-, \ell_-\circ \ell_+, \ell_-^2)$ 
  satisfies the assumptions of Lemma \ref{lem:IFS} (up to conjugating by an appropriate rotation). Indeed, 
  observe first that conjugating by a rotation, we can assume that $\alpha$ is a positive real number. 
Now  let us 
   compute the expression  of $\ell_+^2$: 
   $$\ell^2_+(z) = m^2 z + (m+1)\alpha (1-\abs{m}) = m^2 z+ \beta_{ ++} (1- \abs{m}^2), \text{ where } \beta_{++}  =  \frac{(m+1)\alpha}{\abs{m}+1} .
   $$
Observe first that  $0.98< \abs{m}^2 < 1$.
Next,  we see that $\beta_{++}$ is approximately equal to $\frac{1+i}{2}\alpha$. More precisely, given the assumptions on  $m$
 it can be shown that:
  \begin{itemize} 
  \item    $1.8<\abs{m+1}^2<2.2$, therefore $ 0.6< {\frac{\abs{(m+1)\alpha}}{\abs{m}+1}  }< 0.8$; 
  \item $\abs{\arg{(m+1)} - \frac{\pi}{4}} < \frac{\pi}{20}$.
  \end{itemize}
  For $\ell_+\circ \ell_-$,  $\ell_-\circ \ell_+$ and $ \ell_-^2$, we have analogous estimates, with $\frac{\pi}{4}$ replaced by $-\frac{\pi}{4}$, 
  $\frac{3\pi}{4}$ and $\frac{5\pi}{4}$ respectively. Thus, conjugating by a rotation of angle $\frac{\pi}{4}$,
   we are in position to 
  apply Lemma \ref{lem:IFS}, and the result follows.
  \end{proof}

%
%
%

\subsection{Blenders in $\cd$}

We now study a 2-dimensional version of the phenomenon studied in \S\ref{subs:IFS}.
Similarly to \eqref{eq:A} we define the angular sector 
\begin{equation}\label{eq:A'}
  A' = \set{\alpha \in \cc, \  0.7<\abs{\alpha}< 0.9, \ \abs{\arg(\alpha)} < \frac{\pi}{40}}.  
  \end{equation}
It is easily shown  that for every $\alpha\in A'$, the disk 
 $D\big(\alpha,  \unsur{40}\big)$ is contained in  $A$.

\begin{lem}\label{lem:IFSC2}
Let $d\geq 3$ and $\el = (L_j)_{j=1}^d$ be an IFS in $\dd^2$ generated by biholomorphic
 contractions of the form 
$$ L_j(z,w) = (\ell_j(z), \varphi_j(z,w)),$$ and let $\mathcal E$ be its limit set. 
Assume that  $\ell_j$ is of the form
$\ell_j: z\mapsto m z + \alpha_j (1-\abs{m})
  e^{\frac{2\pi i }{d}j},$ where $0.98<\abs{m} <1$ and  $\alpha_j\in A'$, 
   and $\varphi_j:\dd^2\cv\dd$ is a holomorphic map 
 such that $$\abs{\frac{\fr \varphi_j}{\fr z}} < 1 \text{ and } \abs{\frac{\fr \varphi_j}{\fr w}} < \frac12 .$$
 Then    any  
 vertical graph $\Gamma$  
 intersecting $D(0, \unsur{10})\times \dd$,  with slope bounded by $\frac{1}{100}(1-\abs{m})$ must intersect $\mathcal{E}$,  that is
   $\Gamma\cap \E\neq \emptyset$.

Furthermore, the same holds for any IFS $\overline{\el}$
 generated by $(\overline  L_j)_{j=1}^d$, whenever 
 \begin{equation}\label{eq:C1estimate}
 \norm{L_j - \overline L_j}_{C^1 }< \frac{1}{1000}(1-\abs{m}).
 \end{equation}
\end{lem}

Note that  by the Cauchy estimates a control on the $C^0$ norm of ${L_j - \overline L_j}$ in a slightly larger domain is enough to achieve \eqref{eq:C1estimate}. 

\begin{proof}
Let us    work directly with 
 the  perturbed situation. Let $\mathcal{G}$  be the set of vertical graphs 
$\Gamma$   in $\dd^2$ of the form $z= \gamma(w)$ with     $\abs{\gamma'}<
\unsur{100}(1-\abs{m})$ and $\abs{\gamma(z_0)}<\unsur{10}$ for some $z_0\in \dd$.  

We want to show that if $\Gamma \in \mathcal G$ then 
 $\Gamma\cap \E (\overline\el) \neq \emptyset$. For this we must 
  prove  that there exists an infinite sequence $(j_k)$ 
such that\footnote{Notice the order of compositions.} 
$\Gamma\cap \lrpar{ \overline L_{j_1}\circ \cdots \circ\overline L_{j_k} (\dd^2) } \neq \emptyset$,
 or in other words, that $\overline L_{j_k}\inv \circ\cdots \circ\overline L_{j_1}^{-1}$ is well-defined 
on some piece of $\Gamma$. 

Arguing by induction, to establish this property it is enough to show that if $\Gamma\in \mathcal G$ there exists $j\in \set{1, \ldots , d}$ such that 
$\overline L_j\inv(\Gamma)\cap{\dd^2} \in \mathcal G$. 
Write $\overline L_j = L_j+ \e_j$. 
For this, observe first that the assumption on the slope implies that 
the diameter of the  first projection of $\Gamma$ is smaller than $\unsur{50}(1-\abs{m})$, hence 
 contained in a disk $D\lrpar{z_0, \unsur{50}(1-\abs{m})}$ for some $z_0\in  D\lrpar{0, \unsur{10}}$. Let $j\in  \set{1, \ldots , d}$ be 
 as provided by Lemma \ref{lem:IFS variant} for this value of $z_0$,
  and consider $\overline L_j\inv (\Gamma)$ (for notational ease we drop the restriction to $\dd^2$). 
 This is a subvariety contained in $\overline L_j\inv \lrpar{D\lrpar{z_0, \unsur{50}(1-\abs{m})} \times \dd}$, which, by 
 Lemma \ref{lem:IFS variant} and the bound on $\norm{\e_j}_{C^0}$ is contained in  $D\lrpar{0, \unsur{10}}\times \dd$. 

The  equation of $\overline L_j\inv (\Gamma)$ is of of the form 
\begin{equation}\label{eq:implicit}
\ell_j(z) +\e_1(z,w)=  \gamma\lrpar{\varphi_j(z,w)+\e_2(z,w)} 
\end{equation}
which rewrites as 
 \begin{equation}\label{eq:implicit2}
  z=   \ell_j\inv \lrpar{\gamma\lrpar{\varphi_j(z,w)+\e_2(z,w)} - \e_1(z,w)}.
  \end{equation}
The $z$-derivative of the right hand side of this equation is smaller than 1 so   the contraction mapping principle 
  tells us that it has a unique solution for each $w$. In other words  
    $\overline L_j\inv (\Gamma)$ is a vertical graph. 
    Finally, the implicit function theorem applied to \eqref{eq:implicit} 
    implies
      that the slope of this graph  
         is bounded  by 
         $$\frac{ \abs{\gamma'} \lrpar{\unsur{2} + \abs{\frac{\fr\e_2}{\fr w}}}+ \abs{\frac{\fr\e_1}{\fr w}}}
         {\abs{m} - \abs{\gamma'} \lrpar{1+ \abs{\frac{\fr\e_2}{\fr z}}}- \abs{\frac{\fr\e_1}{\fr z}}}  < \unsur{100}(1-\abs{m}) $$
                  and we are done.
    \end{proof}

\begin{rmk} \label{rmk:blender2}
A similar result holds for an IFS with two branches of the form $(\ell_\pm( z), \varphi_\pm(z,w)$, for $\ell_\pm(z) = mz \pm  \alpha(1-\abs{m})$
and  $m$ close to the imaginary axis,  as  in Lemma \ref{lem:IFS2}. 
Indeed, exactly as   in the proof of that lemma,  it is enough to take two iterates of the IFS to get back to the setting of Lemma \ref{lem:IFSC2}.
\end{rmk}

\subsection{Blenders  for endomorphisms of $\pd$ and the post-critical set}

\begin{thm}\label{thm:blender}
Let $d\geq 3$ and  $f:\pd\cv\pd$ be a product map of the form $$f(z,w)  = (p(z), q(w))=  (p(z), w^d+ \kappa),$$ with $\deg(p) =d$. 
 Assume that  the polynomial $p$
belongs to the bifurcation locus in $\poly_d$ and 
admits  a repelling fixed point $z_0$  of 
low multiplier:  $1<\abs{p'(z_0)} <1.01$. 

 Then there exists a constant $\kappa_0$ depending only on $d$ such that 
 if $\abs{\kappa} > \kappa_0$,    $f$ belongs to the closure of the interior of 
the bifurcation locus in $\hold(\pd)$. 
\end{thm}

It is easy to construct polynomials $p$ satisfying the assumptions of the theorem for every $d\geq 3$.   
A variant of this construction is given  in Theorem \ref{thm:blender2} below which allows to treat   the case $d=2$ as well. 
An interesting open question is whether  $z_0$ can be assumed to be periodic instead of fixed.  

%

\begin{proof} 
We translate the $z$ coordinate so that $z_0 = 0$, and put $m = (p'(0))\inv$, so that $m$ is a complex number such that  
$0.99< \abs{m} <1$.    
To fix the ideas we work in the case $d=3$, the adaptation to the general case is easy and left to the  
 reader.  
 Furthermore to ease notation we replace $\kappa$ by $-\kappa$ so that the second coordinate becomes $q(w)  = w^3- \kappa$.

\medskip

\noindent{\bf Step 1 :} analysis of $J_q$

When $\kappa$ is large, the Julia set of $q$ 
is a Cantor set; here we describe its geometry. 
Consider the disk $D_\kappa := D(0, 2\abs{\kappa}^{1/3})$. The critical point 0 satisfies $q(0) = \kappa\notin D_\kappa$
so $q^{-1}(D_\kappa)$ is the union of 3 topological disks, invariant under a rotation of angle $\frac{2\pi}{3}$.  
Fix a cube root $\kappa^{1/3}$ of $\kappa$. 
If $x\in D_\kappa$, solving the equation $w^3-\kappa =x$ yields
$$w^3 =\kappa + x = \kappa \lrpar{1+\frac{x}{\kappa}} \text{, hence } w 
=    {\kappa}^{1/3} \zeta \lrpar{1 +  \frac{x}{3\kappa} + O\lrpar{   {\abs{\kappa}^{-4/3}}} }$$
where $\zeta \in \set{1,j, j^2}$ ranges over the set of  cube roots of unity. Therefore we see that 
 $q^{-1}(D_\kappa)$ is approximately the union of 3 disks of radius $\frac{2}{3\abs{\kappa}^{1/3}}$ centered at the cube roots of $\kappa$. 
 Define $q_\zeta\inv$ to be the inverse branch of $q$  on $D_\kappa$ such that $q_\zeta\inv(0) = \zeta \kappa^{1/3}$.

We introduce the new coordinate $\tilde w = \unsur{2} \kappa^{-1/3} w$. After coordinate change, the expression of $q$ becomes $\tilde q (\tilde w) = \unsur{2} \kappa^{2/3}(8 \tilde w^3 - 1)$, so that for large enough $\kappa$, 
 $\tilde q\inv (\dd)$ is made of 3 components respectively 
 contained in $D\lrpar{\frac12 \zeta ,\frac12 \abs{\kappa}^{-2/3} }$, $\zeta\in \set{1,j, j^2}$. Similarly, from  the Cauchy estimate 
 we infer that  
 $|{(\tilde q_\zeta\inv)'}|\leq \abs{\kappa}^{-2/3}$ on $\dd$.

\medskip

\noindent{\bf Step 2 :} rescaling and construction of a blender.

Close to the origin, $p$ behaves like the multiplication by $p'(0)$, more precisely we have $p(z) = p'(0)z+ O(z^2)$. 
In particular there is a unique inverse branch $p_0\inv$ defined in a neighborhood of 0 of size $\delta_0(p)$ and such that 
$p_0\inv(0) = 0$.

For $0<\delta \ll \delta_0(p)$ we rescale the 
disk $D(0, \delta)$ to   unit size by introducing the new coordinate  $\tilde z = \delta\inv z$.  In the new coordinate, $p$ becomes 
$\tilde p (\cdot)= \delta\inv  p(\delta \cdot)$, thus 
$\tilde p(\tilde z)  =  p'(0) \tilde z+ O(\delta)$ where the $O(\cdot)$ is uniform for $z\in \dd$, and likewise for the inverse map 
$\tilde p_0\inv (\tilde z) = m \tilde z + O(\delta)$.  Similar results hold for the derivatives by applying the Cauchy estimates in a slightly larger disk.

In the new coordinate system $(\tilde z, \tilde w) = \lrpar{ {\delta\inv}z, \unsur{2}\kappa^{-1/3} w} =: \phi (z,w) $,
 we have an IFS  with 3 branches
 on $\dd^2$ defined by the $(\tilde p_0\inv, \tilde q_\zeta\inv)$ for   $\zeta \in \set{1,j, j^2}$. Its limit set $\tilde \E$ is equal to 
 $\set{0}\times J_q$. 
  Now we perturb $\tilde f$ by introducing 
 $$\tilde f_{_\bullet} (\tilde z, \tilde w) = (\tilde p(\tilde z)- 2 \alpha (1-\abs{m}) \tilde w , \tilde q(\tilde w)),$$
 where $\alpha$ is a real number in the interval $(0.7    , 0.9)$: once for all we choose 
 $\alpha=0.8$. We denote by $\tilde \E_{_\bullet}$ the corresponding limit set.
 Notice that in the original coordinates, the corresponding expression is 
 \begin{equation}\label{eq:feps}
 \phi\inv\circ \tilde f_{_\bullet} \circ \phi : (z,w) \mapsto \lrpar{ p(z) + \e w , q(w)} = f_\e(z,w)
  \text{ for } \e=
  -\frac{\delta \alpha (1-\abs{m})}{\kappa^{1/3}}. 
  \end{equation}
%
The IFS induced by $\tilde f_{_\bullet}$ on $\dd^2$ 
   is induced by 3 inverse branches for  $\widetilde f_{_\bullet}$ of the form 
\begin{align} (\tilde z, \tilde w)\mapsto  \big({\widetilde f_{_\bullet}}\big)_\zeta \inv (\tilde z, \tilde w ) \notag &=
 \left(\tilde p_0\inv\left ( \tilde z + 2\alpha (1-\abs{m}) \tilde
q_\zeta\inv (\tilde w)\right ), \tilde q_\zeta\inv (\tilde w)\right) \\ &=  \lrpar{ m\tilde z + \alpha m (1-\abs{m})  \zeta\lrpar{1+  O\lrpar{{\abs{\kappa}^{-2/3}}}} + O\lrpar{\delta}  ,  
 \frac12 \zeta  +O\lrpar{{\abs{\kappa}^{-2/3}}} }.\label{eq:tildef}
\end{align}
Further conjugating the first coordinate by a rotation of angle $\arg(m)$, we can   assume that the translation part in the first component 
of $\big({\widetilde f_{_\bullet}}\big)_\zeta \inv $ equals $\alpha \abs{m} (1-\abs{m})$ (see also Remark \ref{rmk:rotation}). 
 Hence if $\delta$ is so small   and $\kappa$ so  large 
that \eqref{eq:C1estimate} is satisfied,    
Lemma \ref{lem:IFSC2} applies 
 and we deduce that if $V$ is any vertical graph contained in $D(0, \unsur{10})\times \dd$ and small enough slope, then
 $\tilde V  \cap \tilde{ \mathcal{E}}_{_\bullet}\neq \emptyset$. 
 
 Notice that $\delta$ and $\kappa$ can be chosen independently from each other. By Step 1
  the terms   $O\big({\abs{\kappa}^{-2/3}}\big)$ are actually smaller than 
 ${\abs{\kappa}^{-2/3}}$, hence choosing $\abs{\kappa}\geq 2000^{3/2}$ is enough\footnote{There is no attempt to optimize the bound on $\kappa$ here.}. Likewise, we have to choose $\delta$ so that the term $O(\delta)$ in the first component of \eqref{eq:tildef} is roughly bounded by $\frac{1}{2000}(1-\abs{m})$
 
  
\medskip

\noindent{\bf Step 3.0 : } conclusion in a particular case.

Let $f$ be as in the statement of the theorem, 
and fix $\kappa$ large enough so that   the previous requirements are satisfied. We will first prove the result under the simplifying assumption that there exists a 
simple critical point $c$  such that $p(c) = 0$ (recall that 0 is the moderately repelling fixed point). 


The basic set $\E({f}) = \set{0}\times J_q$ is contained in $J^*$, so this property 
persists in a neighborhood of $f$. Notice that arbitrary close to $f$ there are   maps $ f_1$  for which  $\E({f_1})$ is disjoint from the post-critical set: it is enough to 
consider product maps of the form $( p_1(z) ,q(w))$ for which $p_1$ has  a repelling fixed point  at 0 
not belonging to the post-critical set. It follows that any 
intersection between $\mathcal{E}(g)$ and $g(\crit(g))$ for $g$ close to $f$ is proper in $\hol_3(\pd)$ so it gives rise to bifurcations. 

Consider   the perturbation $f_\e(z,w) = (p(z)+\e w, q(w))$ as   in  \eqref{eq:feps}, 
and let us show that for small enough $\e\neq 0$, 
$f_\e$ lies in the interior of the bifurcation locus in  $\hol_3(\pd)$. 
If $\e$ is small enough then 
$\delta =  \frac{\abs{\kappa}^{1/3}\abs{\e}}{0.8 (1-\abs{m})}$ satisfies the requirements of Step 2, so $\E(f_\e)$ is a blender-type Cantor set contained in 
$D(0, \delta)\times D_\kappa$ (recall that $D_\kappa= D(0, 2\abs{\kappa}^{1/3})$). 

In the rescaled coordinates $(\tilde z, \tilde w$), 
the vertical line $\set{\frac{c}{\delta}}\times \cc$ is a component of the critical set of $\tilde f_{_\bullet}$, and its image   is
 $$\tilde V_{_\bullet}: = \set{(-2\alpha(1-\abs{m}) \tilde w, \tilde q(\tilde w)), \tilde w \in \cc)}. $$ The intersection 
 $\tilde V_{_\bullet} \cap \dd^2$ is the union 
 of 3 vertical graphs of the form $$\tilde z= -2\alpha(1-\abs{m})   \tilde q_\zeta \inv (\tilde w), \ \tilde w\in \dd, \ \zeta\in \set{1,j,j^2},$$
 in particular these graphs are contained 
 in $D\lrpar{0, \unsur{10}}\times \dd$ and their  slope  is  smaller than 
$2(1-\abs{m}){\abs{\kappa}^{-2/3}}$. Thus  by Step 2 we deduce that 
$\tilde V_{_\bullet} \cap {\tilde {\mathcal{E}}}_{_\bullet}\neq \emptyset$, therefore 
$f_\e(\crit(f_\e))\cap \mathcal{E}_\e\neq \emptyset$.

 If now $g$ is close to $f_\e$, as in the proof of Theorem \ref{thm:topology} (see the beginning of Step 3 there), every compact piece of 
$\set{c}\times (\cc\setminus \set{0})$ can be followed as part of $\crit(g)$, so $g(\crit(g))$ contains three vertical graphs in   
$D(0, \delta)\times D_\kappa$ which are close to the corresponding ones for $f_\e$. Likewise, the basic set $\E(g)$ is induced by three 
 inverse branches $g_\zeta^{-1}$ close to  the corresponding ones for $f_\e$, thus Lemma \ref{lem:IFSC2} implies that 
 $\E(g)\cap g(\crit(g))\neq \emptyset$. This shows that  for small $\e\neq 0$,  $f_\e\in \mathring{\bif}$, hence 
$f\in \overline{\mathring{\bif}}$, as desired.

 \medskip

\noindent{\bf Step 3.1:} conclusion in the general case.

Let us now assume that $p$ is an arbitrary polynomial satisfying the assumptions of the theorem. 
Replacing    $p$ by an arbitrary close perturbation, 
we may assume that there exists a simple critical point $c$ and an integer $\ell\geq 1$ such that 
$p^\ell(c) = 0$.  
 Indeed $p$ belongs to the bifurcation locus so it admits an active critical point. We may suppose that  all critical points are simple because the locus of 
polynomials with a multiple critical point is a proper subvariety and the bifurcation locus is not pluripolar. 
Now either 0 is already the image of a critical point and we are done, or we can make a 
   conjugacy depending holomorphically on $p$ such  that $p(1) = 0$ and 0 is fixed. 
Since $c$ is active,   Montel's theorem implies that perturbing $p$ slightly it can be mapped under iteration onto 0 or 1, and we are done. 

Notice that we can also assume that the orbit segment
 $p(c), \ldots, p^\ell(c)$ contains no other critical point: 
otherwise we replace $c$ by the last appearing critical point in this orbit. 

\medskip

As in the case $\ell=1$, for small $\e$ we consider the map 
defined by $f_\e(z,w) = (p(z)+\e w, q(w))$, which admits a blender-type Cantor set $\E(f_\e)$ in 
$D(0, \delta)\times D_\kappa$, for   $\delta =  \frac{\abs{\kappa}^{1/3}\abs{\e}}{0.8 (1-\abs{m})}$. 

We need to show that this Cantor set intersects the post-critical set. Then as before this intersection will be robust under further perturbations and we 
infer 
that $f_\e\in \mathring{\bif}$, hence 
$f\in \overline{\mathring{\bif}}$. 
The difficulty is that we cannot  control 
$f^\ell_\e(\set{c}\times \cc)$ precisely enough to guarantee that it contains an almost flat 
 vertical graph in the rescaled bidisk\footnote{Note that already for $\ell=1$ the variation of 
$f_\e^\ell (\set{c}\times \cc)\cap (D(0,\delta)\times D_\kappa)$ with $\e$ is of  the same  order of magnitude as  $\delta$, and this quantity tends to increase exponentially with $\ell$.}. Instead we will use Proposition \ref{prop:briend duval} together with a graph transform argument. 

\medskip

We start with a lemma, which will be proven afterwards. 

\begin{lem}\label{lem:period3}
For small $\e$   there exists a  holomorphically varying fixed point 
    $x_\e$ for  $f_\e^3$,  contained in 
$\E(f_\e)\cap \lrpar{D\lrpar{0, \frac{\delta}{10}}\times D_\kappa}$. 
\end{lem}

Let us show that there exist arbitrary small non-zero values of $\e$ such that $x_{\e}$ belongs to the  post-critical set of $f_\e$. 
Indeed for small $\e_0>0$, consider the family $(f_\e)_{\e\in D(0, \e_0)}$. For $\e=0$, $x_0$ belongs to $ \set{0}\times \cc = f_0^\ell(\set{c}\times \cc)$. So either 
for every small $\e$, $x_\e\in f_\e^\ell(\set{c}\times \cc)$ and we are done, or the family $(f_\e)_{\e\in D(0, \e_0)}$ admits bifurcations. 
Within this family, the bifurcation locus cannot have isolated points because by \cite[Thm  1.6]{bbd} it supports the measure $dd^cL$ which has continuous potential. It follows that $0\in D(0, \e_0)$ is an accumulation point of the bifurcation locus, so  by Proposition \ref{prop:briend duval} there exists a sequence 
of parameters $0\neq \e_k\cv 0$,    such that $x_{\e_k}$ belongs to the post-critical set.

In addition   $\crit(f_\e) = (\crit(p)\times \cc)\cup (\cc\times \crit(q))$ is a union of vertical and horizontal lines. 
The unique horizontal component $\set{0}\times \cc$ escapes to infinity so we conclude that  for 
such parameters $\e_k$, there exists a critical point $c'$ for $p$ and an integer $N$ 
such that $x_\e\in f_\e^N(\set{c'}\times \cc)$. 
The image of a   graph over $D_\kappa$   in the second coordinate is the union of 3 such graphs: indeed write $\Gamma = \set{z= \varphi(w),  \ w\in D_\kappa}$, 
and observe that 
$$f(\Gamma)  = \bigcup_{\zeta\in \set{1, j, j^2}} \set{z= p(\varphi(q_\zeta^{-1}(w)) + \e q_\zeta^{-1}(w), \ w\in D_\kappa}.$$
Thus we infer that the irreducible component of  $f_\e^n(\set{c'}\times \cc)$ through $x_\e$ is a vertical graph. In particular it is smooth and 
its tangent vector at $x_\e$ is not parallel to the horizontal axis. 
 
 The following lemma then allows to conclude the proof. 
 
 \begin{lem}\label{lem:Wuu}
 The periodic point $x_\e$ of Lemma \ref{lem:period3} admits a strong unstable manifold 
 $W^{uu}(x_\e)$ 
 that is a graph over the second coordinate in $D\lrpar{0, \delta}\times D_\kappa$ with slope 
 smaller than  $\frac{\delta(1-\abs{m})}{200\abs{\kappa}^{1/3}}$.  Furthermore if $\Gamma$ is any any germ of holomorphic disk through $x_\e$, not tangent to the horizontal direction, the sequence of cut-off iterates $f^{3n}(\Gamma)\rest{D\lrpar{0,\delta}\times D_\kappa}$ converges to $W^{uu}(x_\e)$. 
 \end{lem}

For $\e = \e_k$, starting from the component $V$ of $f^N_\e(\set{c'}\times \cc)$ through $x_\e$, we iterate under $f^3_\e$, and the lemma says that for large $n$, 
 $f^{3n}_\e(V)$ contains a graph in $D\lrpar{0, \delta}\times D_\kappa$ of slope less than $\frac{\delta(1-\abs{m})}{200\abs{\kappa}^{1/3}}$,
 intersecting $D\lrpar{0, \frac{\delta}{10}}\times D_\kappa$. In the rescaled coordinates, this corresponds to
  the requirements of Lemma \ref{lem:IFSC2} so we get an  intersection between $f_\e^{N+3n}(\crit(f_\e))$ and $\E(f_\e)$.  
  Notice that the vertical graph producing this intersection is the image under $f_\e^{N+3n}$ of a disk of the form $\set{c'}\times \Delta$, where $\Delta$ is a small 
  topological disk close to a cube root of $\kappa$.  
 Finally, as in the particular case of Step 3.0, if $g$ is a small perturbation  of $f_\e$, $\set{c'}\times \Delta$ can 
 be lifted to a disk  $\Delta(g)\subset 
 \crit(g)$, and $g^{N+3n}(\Delta(g))$ is $C^1$-close to the corresponding component of $f_\e^{N+3n}(\crit(f_\e))$, therefore it intersects $\E(g)$. 
  This completes the proof of the theorem.   
  \end{proof}

Before proving Lemma \ref{lem:period3} let us state a  Rouché-like fixed point theorem in two variables:

\begin{prop}\label{prop:rouche}
Let $h:N\big(\overline{\mathbb{B}}\big) \cv \cd$ be a holomorphic map defined in a neighborhood of a 
  bidisk in $\cd$. 
Assume that $h$ admits a unique simple fixed point in $\mathbb{B} $. If 
$\eta:N\big(\overline{ \mathbb{B} }\big) \cv \cd$ is a holomorphic map such that 
$\norm{\eta} < \norm{h-\mathrm{id}}$ on $\fr\bb $ then  $h+\eta$ admits a unique fixed point in $\bb$.
\end{prop}

\begin{proof}
Consider the continuous family of holomorphic mappings $F_t = F+t\eta$ for $t\in [0,1]$. For $t = 0$, the equation $F_t(z) = z$ admits a unique simple solution in
$\mathbb{B}$.  The assumption on $\eta$ implies that for every $\eta$ in $[0,1]$, 
$\norm{F_t(z) - z}>0$ on $\fr\bb$. Thus the continuity of intersection indices of properly intersection varieties of complementary dimensions implies the result. 
\end{proof}

\begin{proof}[Proof of Lemma \ref{lem:period3}]	
To understand the argument, 	let us  come back to the linear IFS studied in Lemma \ref{lem:IFS} (for $d=3$). 
Each $\ell_j$ admits a unique fixed point, and there are two cases.
 Either $m$ is far away from 1 and this fixed point is close to the origin. Going to two dimensions,
  we will choose  $x_\e$ corresponding to one of these fixed points. 
 When  $m$ is close to 1 the fixed point is 
   close to the boundary of the unit disk, so it is not convenient for us. On the other hand if we look at  $\ell_1\circ \ell_2\circ \ell_3$, 
   the translation terms almost compensate and we get a fixed point close to the origin, now corresponding to a period 3 point for $f_\e$. 

\medskip

For the details it is convenient to work in the rescaled coordinate system $(\tilde z, \tilde w)$:
 the expression of $\tilde f_{_\bullet}\inv$ is given in \eqref{eq:tildef} (hereafter we drop the subscript   for notational convenience),  
 and we look for 
a fixed point (resp. a period 3 point) in $D(0, \unsur{10})\times \dd$.  Assume first $m$ is far away from 1: $\abs{m-1}> \unsur{10}$ and consider the inverse branch $
\tilde f_1\inv$. 
 We write  $\tilde f_1\inv(\tilde z, \tilde w)$ as the sum of an affine  term and a perturbation : 
$$\tilde f_1\inv(\tilde z, \tilde w)  = h(\tilde z, \tilde w) +   \eta(\tilde z, \tilde w) \text{ , where } h(\tilde z, \tilde w) =  \lrpar{ m \tilde z + \alpha m (1-\abs{m}), \frac12} .$$
Solving $h(\tilde z, \tilde w) = (\tilde z, \tilde w)$ yields the solution 
$ \big( \frac{\alpha m(1- \abs{m})}{m-1}, \unsur{2}\big)$ which belongs to 
$D(0, \unsur{10})\times \dd$ because 
 $$\abs{\frac{\alpha m(1- \abs{m})}{m-1}} \leq \frac{0.8 /{100}}{1/10} = \frac{8}{100}$$
  (recall $\alpha = 0.8$ and $1-\abs{m}<0.01$). 
On the other hand if $\abs{\tilde z} = \frac{1}{10}$ we get 
$$\norm{h(\tilde z, \tilde w)- (\tilde z, \tilde w) }  \geq  \abs{(m-1) \tilde z + \alpha m (1-\abs{m})} 
\geq   \frac{1}{10} \abs{m-1}  - \frac{0.8}{100}
\geq  \frac{2}{1000} , $$ and  when $\abs{\tilde w}=1$ considering the second component gives $\norm{   h(\tilde z, \tilde w)  - (\tilde z, \tilde w) }\geq 1/3$,
 so 
$$ \norm{   h(\tilde z, \tilde w)  - (\tilde z, \tilde w) }     \geq   \frac{2}{1000} \text{ on } \fr \big(D\big({0, \unsur{10}}\big) \times \dd\big).$$
Finally the choices already made for $\kappa$ and $\delta$
 imply that $\norm{\eta(\tilde z, \tilde w)}<  \frac{2}{1000}$    
 and Proposition \ref{prop:rouche} yields the desired fixed point.  (Recall that Theorem \ref{thm:blender} claims a uniformity in $\kappa$; 
 on the other hand $\delta$ can  be freely reduced)

\medskip

Suppose now that $\abs{m-1}\leq  \unsur{10}$, and consider the inverse branch for $\tilde f ^3$ given by 
$\tilde f_{j^2}\inv\circ\tilde f_j\inv \circ\tilde f_1\inv$. After computation, it expresses in coordinates as 
\begin{equation}\label{eq:branch}
\tilde f_{j^2}\inv\circ\tilde f_j\inv \circ\tilde f_1\inv(\tilde z, \tilde w) = \lrpar{ m^3 \tilde z + \alpha m(1-\abs{m}) (j^2+ mj+ m^2) +\eta_1(\tilde z, \tilde w), 
 \frac{j^2}{2} + \eta_2(\tilde z, \tilde w) },$$
 with   
$$\abs{\eta_1(\tilde z, \tilde w)} \leq 3 (1- \abs{m} ) \abs{\kappa}^{-2/3}  +3 O(\delta)
 \text{ and } \abs{\eta_2(\tilde z, \tilde w)}\leq   \abs{\kappa}^{-2/3}
  .\end{equation}
In this case we show directly that $\tilde f_{j^2}\inv\circ\tilde f_j\inv \circ\tilde f_1\inv$ sends  the bidisk $D\big({0, \unsur{10}}\big) \times \dd$ strictly into itself, 
 so it admits a unique fixed point by contraction of the Kobayashi metric. By the maximum principle it is enough to show that 
 the boundary of the bidisk is mapped into its interior, and  clearly  we only need to focus on the first coordinate. 
 For $\abs{\tilde z}= \frac{1}{10}$ the first component in
 \eqref{eq:branch}
  is bounded by 
 \begin{align*}
  \frac{\abs{m}^3}{10} + \alpha (1-\abs{m})\abs{\frac{m^3-1}{m-j}} + \abs{\eta_1(\tilde z, \tilde w)} 
  &\leq  \frac{\abs{m}^3}{10}  + 2 (1-\abs{m})^2  +  \unsur{100}(1-\abs{m})\\
  &\leq \frac{\abs{m}}{10} +  \frac{3}{100} (1-\abs{m}) < \frac{1}{10},
  \end{align*} 
 where in the first line we use 
 $$
  \abs{\frac{m^3-1}{m-j}} \leq   \abs{\frac{(1-m)(m^2+m+1)}{m-j}}\leq   \frac{3(1-\abs{m})}{\abs{1-j}- \abs{m-1}} \leq \frac{3}{\sqrt{3}-1/10}(1-\abs{m}) \leq 2(1-\abs{m})$$
 and the bound for $ \abs{\eta_1(\tilde z, \tilde w)}$   coming from our choice of $\kappa$ and $\delta$ in Step 2. Thus the desired contraction property is established and the result follows. 
\end{proof}

\begin{proof}[Proof of Lemma \ref{lem:Wuu}]
The existence and the graph transform property of the strong unstable manifold are classical. In our case the specific geometric features of $f_\e$ make the construction rather easy so we sketch it for convenience. 
Since $  f_\e$ preserves the foliation $  \set{w = \cst}$, corresponding to the least repelling direction, $x_\e$ admits a weak unstable manifold  contained in a horizontal leaf. On the other hand, if $g$ denotes the inverse branch of $  f^3_\e$  such that ${x_\e} = g(x_\e)$, then $g^n(D(0, \delta)\times D_\kappa)$  
is a sequence of topological bidisks converging to $\set{x_\e}$, which are asymptotically stretched in the horizontal direction. 
Therefore if $\Gamma$ is a germ of holomorphic disk through $x_\e$ transverse to the horizontal leaf, for large enough $n$ 
it crosses  $g^n(D(0, \delta)\times D_\kappa)$ vertically, so iterating forward, the cut-off iterate $$f_\e^{3n}(\Gamma)\rest{D(0,\delta)\times D_\kappa} = 
f_\e^{3n}(\Gamma \cap g^n(D(0, \delta)\times D_\kappa))$$ is a vertical graph in $D(0, \delta)\times D_\kappa$. Furthermore, it $\Gamma$ and $\Gamma'$ are two such graphs, then  the $C^0$ distance between 
$\Gamma$ and $\Gamma' $ in 
$ g^n(D(0, \delta)\times D_\kappa)$  is $O(\abs{\kappa}^{-n})$, and it gets  multiplied by  a factor $O(\abs{m}^{-3n})$ under $f_\e^{3n}$. Applying this to $\Gamma$ and 
$f_\e^3(\Gamma)$ shows that $f_\e^{3n}(\Gamma)\rest{D(0,\delta)\times D_\kappa} $ is Cauchy, hence converges, and its limit is by definition the strong unstable 
manifold $W^{uu}(x_\e)$.

It remains to establish the estimate on the slope of $W^{uu}(x_\e)$. For this it is more convenient to work in the rescaled coordinates $(\tilde z, \tilde w)$, in which case the   expected bound on the slope is $\frac{1}{100}(1-\abs{m})$. Since in these coordinates, the strong unstable manifold is the only vertical graph invariant under the graph transform, it is enough to show that the set of vertical graphs through $x_\e$
with slope bounded by $\frac{1}{100}(1-\abs{m})$ is graph transform invariant. As already seen, if 
$\tilde z  = \varphi(\tilde w)$ is   a vertical graph in $\dd\times \dd$, its forward image under $\tilde f_\e$, 
restricted to $\cc\times \dd$ is the union of 3 vertical graphs of equation 
$$\tilde z  = \tilde p (\varphi(\tilde q_\zeta\inv(\tilde w))) -2\alpha (1-\abs{m}) q_\zeta\inv(\tilde w).$$ Assuming $\norm{ \varphi'}_\dd\leq \frac{1}{100}(1-\abs{m})$,  each of them has a slope bounded by 
$$\big\| (q_\zeta\inv)' \big\|_\dd \norm{\tilde p'}_\dd \norm{ \varphi'}_\dd+ 2   (1-\abs{m})\big\| (q_\zeta\inv)' \big\|_\dd \leq(1-\abs{m}) \lrpar{ \frac1{100}\abs{\kappa}^{-2/3} 
\norm{\tilde p'}_\dd + 2\abs{\kappa}^{-2/3} }.  $$ Since for small $\delta$, $ \norm{\tilde p'}_\dd \leq 2$ and $\abs{\kappa}^{2/3}\geq 2000$, 
this quantity is bounded by $\frac{1}{100}(1-\abs{m})$.

For a general graph $\Gamma$ we cannot iterate this reasoning because its forward iterates may leave the bidisk. However 
 in our case we start with $\Gamma\ni x_\e$  which is either  fixed or of period 3 (with its orbit contained in $\dd^2$) so we indeed get an invariant set of vertical graphs and we are done. 
\end{proof} 
 
 The uniformity in $\kappa$ in Theorem \ref{thm:blender} allows to let the multiplier tend to 1.

 \begin{cor}\label{cor:parabolic}
 Let $d\geq 3$ and  $f:\pd\cv\pd$ be a product map of the form $$f(z,w)  =   (p(z), w^d+ \kappa).$$
 Assume   $p_0$ admits a neutral fixed point. 
 Then there exists a constant $\kappa_0 = \kappa_0(d) $ such that 
 if $\abs{\kappa} > \kappa_0$,    $f$ belongs to the closure of the interior of 
the bifurcation locus in $\hold(\pd)$. 
 \end{cor}
 
 \begin{proof}
Using  the uniformity with respect to  $\kappa$ in  Theorem \ref{thm:blender} 
it is enough  to show that in every neighborhood of $p$ in $\poly_d$ there exists a polynomial $p_1$ with a repelling 
 fixed point of low multiplier, belonging to the bifurcation locus. 
 
 Without loss of generality we may assume that $z_0$ is rationally indifferent, that is  
 $p'(z_0) = e^{2i\pi \frac{p}{q}}$. Taking a branched cover of $\poly_d$ if necessary (this is needed  only  if $p'(z_0)= 1$), we can 
 follow the fixed point $z_0$ holomorphically and normalize the coordinates so that $z_0 = 0$. Fix a one-dimensional 
 holomorphic family of polynomials $(p_\la)_{\la\in \dd}$ with $p_0 = p$ and  such that  $\frac{d  }{d\la}\lrpar{p'_\la(0)}
 \rest{\la = 0}    \neq 0$. Put   $\rho_\la = p'_\la(0)$. Then a classical computation shows that 
 there exists a local change of coordinates $x = \varphi_\la(z)$ depending holomorphically on $\la$ such that in the new coordinates $f_\la^q$ expresses as 
 $$f_\la^q(x)   = \rho_\la^q x  + x^{\nu q + 1} + x^{\nu q  + 2} g_\la(x)$$ for some integer $\nu \geq 1$ (see \cite[Prop. 1]{tv} or \cite[prop. 8.1]{tangencies}). 
 Write $\rho_\la^q =  1+ b\lambda + O( \la^2)$, where $b = q \frac{d \rho_\la}{d\la}\rest{\la = 0}$. Then 
\begin{equation}\label{eq:fixed}
f_\la^q(x) - x  =  x \lrpar{(\rho_\la^q-1) + x^{\nu q} +   x^{\nu q+1 }g_\la(x)}
\end{equation}
 hence $f_\la^q$ admits $\nu q+1$ fixed points near the origin: one at 0 with multiplier $\rho_\la^q$ and the remaining $\nu q$ ones approximately equal to the $(\nu q)^{\rm th}$ roots of $(1- \rho_\la^q)$. More precisely plugging into the right hand side of 
  \eqref{eq:fixed}, we get that  these fixed points $(\alpha_i)_{i=1, \ldots, \nu q}$ satisfy 
 $$(\alpha_i)^{\nu q} = (1- \rho_\la^q) - (\alpha_i)^{\nu q+1} g_\la(\alpha_i) = -b\la + O\lrpar{\abs{\la}^{1+ \frac{ 1}{\nu q}}},$$
 and their multipliers are of the form 
 $$(f_\la^q)'(\alpha_i) = \rho_\la ^q  + (\nu q+ 1) \alpha_i ^{\nu q} + O \lrpar{(\alpha_i)^{\nu q+1}} = 1- \nu q b \la +  O\lrpar{\abs{\la}^{1+ \frac{ 1}{\nu q}}}. $$
We see that arbitrary close to the origin in parameter space, in the regime where $0$ is repelling under $f_\la^q$ (i.e. $\abs{\rho_\la^q}>1$), 
we can arrange so that $\abs{(f_\la^q)'(\alpha_i) }<1$ (e.g. when  $b\la$ is real and positive) or $\abs{(f_\la^q)'(\alpha_i) }>1$ (by taking $b\la$ on the imaginary axis). Therefore bifurcations happen and the 
result follows.
 \end{proof}

We now state a version of  Theorem \ref{thm:blender} in degree 2.   

\begin{thm}\label{thm:blender2}
There exists a parameter $c$ in the Mandelbrot set such that for every large enough    $\kappa\in \cc $, the product map 
$f(z,w) = (z^2+c, w^2+\kappa)$ belongs to the closure of the interior of the bifurcation locus in $\hol_2(\pd)$. 
\end{thm}
 
\begin{proof}[Proof (sketch)] 
Pick $c\in M$ such that $z^2+c$ has a fixed point $z_0$ whose multiplier $2z_0 =: m\inv$  satisfies $0.99 <\abs{m} <1$ and $\abs{\arg(m) -\frac{\pi}{2}}< \frac{\pi}{50}$. 
(The existence of such a parameter follows exactly as in the previous corollary, by starting from the boundary of the main cardioid and perturbing in the appropriate direction.)  
Shifting $c$ slightly we can further arrange that the critical point $0$ falls onto $z_0$ under iteration.  
Arguing as in Step 2 of the proof of Theorem  \ref{thm:blender} (using Lemma \ref{lem:IFS2} and Remark \ref{rmk:blender2}) shows that 
the perturbation $f_\e(z,w) = (z^2+c+\e w, w^2+\kappa)$ which  
  admits  a blender in $D(z_0, \delta)\times D({0, 2\abs{\kappa}^{1/2}})$ for $\delta = \frac{\abs{\kappa}^{1/2} \abs{\e}}{0.95(1-\abs{m})}$, 
  when $\e$ is small and $\kappa$ is larger than some absolute constant. Finally the argument given in the third step of the proof  
   implies the existence of a sequence $\e_j\cv 0$ such that the post-critical set of $f_{\e_j}$ 
   robustly intersects this blender, and the result follows.    We leave the reader fill   the details.
   \end{proof}

  \section{Further considerations and open problems}\label{sec:further}
  
  \subsection{Bifurcations of saddle sets}\label{subs:newhouse}
%
%
%
%
In this paragraph we show that  persistent homoclinic tangencies, hence robust homoclinic bifurcations, 
 can coexist with $J^*$-stability.   This implies that the robust bifurcations constructed in this paper are {\em not} induced by the Newhouse phenomenon. This also highlights 
the large gap between  $J^*$-stability and structural stability on $\pd$. 

\begin{thm}\label{thm:newhouse}
There exists a $J^*$-stable family 
of holomorphic endomorphisms of $\pd$ whose members possess 
 a horseshoe with  a generic tangency between the  stable and unstable laminations. 
 \end{thm}

\begin{proof}
The idea is to embed a polynomial automorphism $h$ of $\cd$
  with a robust  homoclinic  tangency in an endomorphism of $\pd$ (this idea  already 
 appears in Buzzard \cite{buzzard} and Gavosto \cite{gavosto}). What we need to do is to arrange so that $J^*$-stability holds. 
 
 \medskip 
 
%
   
 The first observation is that we can choose $h$ to be a product of two complex Hénon mappings   of the same degree. Indeed, the  automorphisms constructed by Buzzard are of the form $F = F_3\circ F_2 \circ F_1$, where $F_1(z,w) = (z+f(w), w)$, $F_2(z,w)  = (z, w+g(z))$ and $F_3 (z,w)= (c z, c\inv w)$ 
 (see \cite[p.394]{buzzard}). At the initial stage of Buzzard's argument, $f$ and $g$ are holomorphic mappings defined on certain open
  subsets of $\cc$, which are later approximated by polynomials. Thus we can assume that $f$ and $g$ are polynomials of the same (large)
   degree $d$. 
 Let $\iota$ be the involution $(z,w)\mapsto (w,z)$   
and write $F$ as 
 $(F_3\circ F_2\circ \iota)\circ (\iota\circ F_1)$. Then 
 $\iota \circ F_1$ is a conservative Hénon map, but   $F_3\circ F_2\circ \iota$ is not of the right form. So we 
  introduce the  linear map $\ell: (z,w)\mapsto (z,c w)$, and write 
  $$F=(F_3\circ F_2\circ \iota\circ \ell\inv )\circ (\ell  \circ \iota\circ F_1) := h^-\circ h^+.$$ We 
  leave the reader check that $h^-$ and $h^+$ are of the form $(z,w)\mapsto (w, c^{\pm 1} z+p^\pm(w))$, with $p^\pm$ of degree $d$, as desired. 
    
  \medskip
  
  With notation as above introduce $F_\e = h^-_\e\circ h^+_\e$, where $$h^\pm_\e (z,w) = (w+ \e z^d , c^{\pm 1} z+p^\pm(w)).$$ This is a holomorphic endomorphism of $\pd$ which on every given compact subset in $\cd$ 
   can be seen as a small perturbation of $F$.    In particular if we fix $\e$   small enough and  let  $c$ and $p^\pm$ vary we get a holomorphic family of endomorphisms of $\pd$  with a saddle set exhibiting a persistent generic tangency between its stable and unstable laminations. 
   
 What remains to do is to show that this family is $J^*$-stable.  This will be a consequence of the  following lemma. 
 
 \begin{lem}\label{lem:heps}
 Let $h_\e(z,w)  =(w+ \e z^d, c z + p(w))$, where $p$ is a polynomial of degree $d$.  Let $\beta = \unsur{ d-1}$, 
 $\alpha$ be a real number such that
 $ \beta/d <\alpha < \beta$ and  
 $$V_\e = \set{(z,w)\in \cd, \unsur{2} \e^{-\beta} <\abs{z} <  \frac32  \e^{-\beta}, \ \abs{w} < \e^{-\alpha}}.$$
  Then for  $\e$ sufficiently small (locally uniformly with respect to the parameters $c$ and $p$),
   $h_\e\inv(V_\e)\Subset V_\e$ and   $h_\e : h_\e\inv(V_\e)\to V_\e$ 
  is an unbranched  covering.
  \end{lem}
   
     Applying this lemma successively to $h_\e^+$ and $h_\e^-$, we deduce  that for small $\e$, $F_\e\inv (V_\e)\Subset V_\e$ and 
   $F_\e : F_\e\inv(V_\e)\to V_\e$ is a covering of degree $(2d)^2$. Hence 
   $J^*(F_\e) \subset V_\e$, and since  in addition  preimages in $V_\e$ can be followed locally 
   holomorphically with the parameters, it follows that the family obtained by fixing $\e$ and varying $c$ and $p^\pm$ 
   is locally $J^*$-stable.
   \end{proof}

\begin{proof}[Proof of Lemma \ref{lem:heps}]
To ease notation, without loss of generality we assume that $c=1$. We rescale the  coordinates by putting $\tilde z = \e^{-1/(d-1)}z  = \e^{-\beta} z$ and 
$\tilde w = w$. In the new coordinates (still denoted by $(z,w)$ for convenience) $h_\e$ becomes
$$ h_\e: (z,w) \mapsto \lrpar{\e^\beta   w+    z^d, \e^{-\beta}   z + p\lrpar{  w}}  $$ 
and $V_\e$ becomes $A(1/2, 3/2)\times D(0, \e^{-\alpha})$, where $A(1/2, 3/2)$ is the open annulus bounded by the circles of radii 1/2 and 3/2. 

Let now $(u,v)\in V_\e$ and $(z,w) $ be such that $h_\e(z,w)  = (u,v)$, that is:
$${ \it (i)}  \; 
   \e^\beta   w+    z^d = u \text{  and } {\it (ii)} \;  \e^{-\beta}   z + p\lrpar{  w} = v. $$
Then we claim that $\e^\beta   \abs{w} \ll    \abs z^d$. Indeed otherwise since $\abs u\approx  1$ we get that 
$\e^\beta  \abs w \gtrsim 1$ hence $\abs w \gtrsim \e^{-\beta}$, therefore $\abs {p(w)} \gtrsim \e^{-d\beta}$. Since $d\beta > \alpha$, this implies that 
$\abs {p(w)} \gg \abs v $ hence by {\em (ii)} $\e^{-\beta} \abs{z}  \approx \abs {p(w)}\approx \abs w^d\gtrsim \e^{-d\beta}$, so 
$\abs{z}  \gtrsim \e^{-\beta (d-1)}  = \e^{-1}$. Plugging this back into {\em (i)}, we see that   $\e^\beta   \abs{w} \approx   \abs z^d$. Together with 
$\e^{-\beta} \abs{z}\approx \abs w^d$, this yields $\abs z \approx \e^{\beta/(d+1)} = \e^{1/(d^2-1)}   = o(1)$. 
This is contradictory, so the claim is proved. 

Since $\e^\beta   \abs{w} \ll    \abs z^d$, the equation {\em (i)} admits $d$ unramified solutions in $z$, 
close to the $d^{\rm th}$ roots of $u$. Therefore  $\e^{-\beta} \abs{z} \approx \e^{-\beta} \gg \abs{v}$ so for each such $z$,
solving {\em (ii)} in the variable 
  $w$ gives $d$ solutions satisfying 
   $\abs w ^d\approx \e^{-\beta}$, that is $\abs w \approx \e^{-\beta/d}$, and since $\beta/d < \alpha$  these solutions belong to $D(0, \e^{-\alpha})$.
   
Finally, the critical set is the curve of equation $dz^{d-1}p'(w) =1$. If $(z,w)\in \crit(h_\e)\cap V_\e$ we have    
$\abs{z}\approx 1$ so  we get that $\abs{w} = O(1)$. Thus the second coordinate of $h_\e(z,w)$ is of order of magnitude
$\e^{-\beta} \gg \e^{-\alpha}$. This means that critical points escape $V_\e$ after  one iteration (i.e. $h_\e(\crit(h_\e)\cap V_\e)\cap V_\e = \emptyset$) 
and we conclude that $h_\e$ is unbranched in $h_\e\inv(V_\e)$.
\end{proof}

\begin{rmk}
Another consequence of Theorem \ref{thm:newhouse}
 is that there exists a $J^*$-stable family whose generic members have 
 with infinitely many repelling periodic points outside $J^*$. 
Indeed since the polynomial automorphism 
 $F$ is conservative, by a small perturbation we can choose the Jacobian to be either smaller of larger than 1. In the latter case, it is well known that
 the persistent homoclinic tangency gives rise to a residual set of parameters with infinitely many sources.
\end{rmk}

     \subsection{Higher dimension}\label{subs:higher}
So far we have concentrated on complex dimension 2. The results of Section \ref{sec:blender} can actually be adapted to an arbitrary number of dimensions. 
Let us only state one result, which guarantees the existence of open sets in the bifurcation locus in $\hold(\pk)$ for every $d\geq 2$ and every $k\geq 3$. 

\begin{thm}\label{thm:higherdim}
Let $f$ be a polynomial mapping in $\cd$ of the form $f(z,w)  = (p(z), w^d+\kappa)$, where $p\in \poly_d$ admits a fixed point $z_0$
of multiplier $1<\abs{p'(z_0)}<1.01$ and belongs to the bifurcation locus in $\poly_d$ (if $d=2$ we further require that $\abs{\arg({p'(z_0)})- \pi/2}<\pi/50$). 

Let $g$ be a regular polynomial mapping in $\cc^{k-2}$ admitting a repelling fixed point in $J^*$ with eigenvalues 
 larger than 2. 

Then if $\abs{\kappa}> \kappa_0(d,k)$ is sufficiently large, 
 the product   map $(f,g)$   belongs to  the closure of the interior of the bifurcation locus  in $\hold(\pk)$. 
 \end{thm}

\begin{proof} 
We focus on the case $d=3$ and leave the adaptation to  degree 2 to the reader. 
We keep notation as in the proof of Theorem  \ref{thm:blender}  and  explain how to generalize the argument. 
Put  $x = (z,w)$ and denote   by $y$ the  variable in
 $\cc^{k-2}$. A preliminary observation is that 
$F(x,y):=(f(x),g(y))$   defines a regular polynomial mapping in $\cc^k$ so it extends holomorphically to $\pk$.  Let $y_0$ be the repelling fixed point for $g$ 
considered in the statement of the theorem, and recall that $\E(f)  = \set{0}\times J_{w^d+\kappa}$ is a basic repeller contained in $J^*(f)$. Therefore $\E(F)  := 
\E(f)\times \set{y_0}$ is a basic repeller contained in $J^*(F)$ and this property persists for its  continuation after  a small perturbation of $F$ in $\hold(\pk)$ by Lemma  \ref{lem:continuation}. 

Let $f_\e(x) = f_\e(z,w) = (p(z)+\e w, w^d+ \kappa)$ as before, and put $F_\e (x,y ) = (f_\e(x), g(y))$. 
Recall from Theorem \ref{thm:blender}  that after translation and rescaling of the $(z,w)$ coordinate, when $\e$ is small and $\kappa$ is large, 
the continuation of 
 $ \set{0}\times J_q$ defines  a blender type repeller for $ f_\e$  in the unit bidisk. 
 By assumption, there exists $\eta>0$ such that $g$ admits a contracting 
  inverse branch mapping $y_0$ to itself, of derivative norm smaller than $1/2$
 in the polydisk centered at $y_0$  and  of radius $\eta$. 
 We translate and rescale the $y$ coordinate by putting $\tilde y  = \eta\inv (y - y_0)$. In the new coordinates $(\tilde x, \tilde y)$, the resulting map 
 $\tilde F_\e$ defines an IFS 
 on $\dd^k$ with $d$ branches for every small enough $\e$. 
 
 \medskip
 
 The mechanism leading to robust intersections between $\E(F)$ and the post-critical set will be based on the following higher dimensional version of Lemma \ref{lem:IFSC2}, which is worth stating precisely. 
 
  \begin{lem}\label{lem:IFSCk}
Let $d\geq 3$ and $\el = (L_j)_{j=1}^d$ be an IFS in $\dd^k$ generated by biholomorphic
 contractions of the form 
$$ L_j(z,\omega) = (\ell_j(z), \varphi_j(z,\omega)),  $$ with $(z,\omega) \in \dd\times  \dd^{k-1}$,   and let $\mathcal E$ be its limit set. 
Assume that  $\ell_j$ is of the form
$$\ell_j( z)= m z + \alpha_j (1-\abs{m})
  e^{\frac{2\pi i }{d}j},\text{  where } 0.98<\abs{m} <1\text{ and  }\alpha_j\in A' \text{ (see \eqref{eq:A'}),}$$ 
   and $\varphi_j:\dd^k\cv\dd^{k-1}$ is a holomorphic map 
 such that  $ \norm{ {\fr_z}\varphi_j }< 1$ and  $\norm{ {\fr_\omega}\varphi_j }< \frac12$. 
 
 Then    any  
 vertical graph $\Gamma$  of the form $z = \gamma(\omega)$ 
 intersecting $D(0, \unsur{10})\times \dd^{k-1}$, whose slope satisfies  $ \norm {d\gamma}\leq \frac{1}{100\sqrt{k-1}}(1-\abs{m})$   intersects $\mathcal{E}$. 
 
Furthermore, the same holds for any IFS $\overline{\el}$
 generated by $(\overline  L_j)_{j=1}^d$, whenever the $C^1$ norm of 
 ${L_j - \overline L_j}$ is bounded by $\frac{1}{1000\sqrt{k-1}}(1-\abs{m}).$
 \end{lem}
 
 The proof is identical to that of Lemma \ref{lem:IFSC2}. The appearance of the factor $\sqrt{k-1}$ in the constants comes from the fact 
 that the diameter of the first projection of $\Gamma$  is now bounded by  $  \mathrm{diam}(\dd^{k-1})\times \mathrm{slope}(\Gamma)
  = 2\sqrt{k-1} \norm{d\gamma}$.
 Under the above assumptions on $\e$, $\kappa$ and $\eta$, in the rescaled coordinates (and possibly after rotation in $z$), the IFS  induced by $F_\e$ on 
 $\dd^k$  satisfies the hypotheses of the lemma (see \eqref{eq:tildef}).

 \medskip
 
 The third step of Theorem \ref{thm:blender} shows that, possibly   after a preliminary perturbation of $p$,
   there exists a sequence of parameters $\e_j\cv 0$ and for every  $k$
   an integer $N' = N+3n$ and a   component   of multiplicity 1
    of $\crit{f_{\e_j}}$ of the form $\set{c'}\times \cc$, 
 such that (after coordinate rescaling) $f_{\e_j}^{N'}(\set{c'}\times \cc)$ contains a   vertical graph  $\Gamma$ 
 in the unit bidisk, satisfying the assumptions of Lemma \ref{lem:IFSC2}. Remark that 
 by Lemma \ref{lem:Wuu}, by increasing 
 $\kappa$ we can make the slope of this graph  smaller than $\frac{1}{100\sqrt{k-1}}(1-\abs{m})$. 
 
   Going to   dimension $k$,  
    $\set{c'}\times\cc^{k-1}$ is a  component of multiplicity 1 of $F_\e$  and the previous discussion shows that  (after rescaling)  
 $F_{\e_j}^{N'}(\set{c'}\times\cc^{k-1})\cap \dd^k$ admits  a component 
 of the form $\Gamma\times \cc^{k-2}$, hence satisfying the requirements of Lemma 
 \ref{lem:IFSCk}.  In addition, this vertical graph is the image of a piece of  $\Gamma\times \cc^{k-2}$
 contained in  $\set{c'}\times \Delta \times U$, where $\Delta$ is a disk in $\cc$ close to a cube root of $\kappa$ and 
 $U$ is a small neighborhood of the repelling point $y_0$. 
Increasing $N'$ further if necessary, we can assume that $U$ is disjoint from $\crit(g)$. Since $\crit(F_\e) = \pi_1\inv(\crit (f))\cup    
 \pi_2\inv(\crit (g))$, we infer that $\crit(F_\e)$ is smooth along $\set{c'}\times \cc \times U$. 
 
 We are now ready to conclude that $F_{\e_j}\in \mathring{\bif}$. Indeed 
  if $G$ is a small perturbation of $F$ in 
 $\hold(\pk)$, we can lift $\set{c'}\times \Delta \times U$ to an open cell in $\crit(G)$, whose image $\Gamma(G)$ 
 under $G^{N'}$ is close to  
 $\Gamma\times \cc^{k-2}$. So Lemma \ref{lem:IFSCk} implies that $\Gamma(G)\cap \E (G)\neq \emptyset$, and we are done. 
 \end{proof}

  \subsection{Open problems}\label{subs:open}
The constructions of Section \ref{sec:blender} raise a number of interesting problems. As already said,  a natural question after 
Theorem \ref{thm:blender} (and Corollary \ref{cor:parabolic}) is whether in these statements the 
fixed point $z_0$ can be assumed to be periodic instead of fixed. 
Using the technique of Corollary \ref{cor:parabolic} one could then simply replace the hypothesis on $p$ in Theorem \ref{thm:blender}
 by ``$p$ belongs to the bifurcation locus". 
The difficulty is that 
our technique relies on perturbations of a specific form like $(p(z)+\e w, q(w))$, which does not {\em a priori}  behave nicely under iteration.

More generally we believe that the following is true:

\begin{conjecture}
Let $f(z,w) = (p(z), q(w))$ be a  product of two polynomial maps of degree $d$ in $\cc$. Then if $p$ or $q$ belongs to the bifurcation locus in $\poly_d(\cc)$, 
$f$ belongs to the closure of the interior of the  bifurcation locus in $\hold(\pd)$.
\end{conjecture}
 
In other words, the bifurcation locus restricted to the subfamily $\poly_d\times \poly_d$
 has empty interior, nevertheless it should be accumulated at every point by robust bifurcations in the larger space 
of holomorphic mappings on $\pd$. 
  
  \medskip
  
  In dimension 1, Shishikura's   theorem on the dimension of the boundary of the Mandelbrot set  
   \cite{shishikura}, as well its generalizations to arbitrary spaces of rational mappings by Tan Lei and McMullen \cite{tan lei, mcm universal} 
  are based on the construction of Cantor repellers  of large dimension (i.e. close to  2). 
  In dimension 2, bifurcations are created by collisions between the 
  post-critical set (which has Hausdorff dimension 2) and hyperbolic repellers. 
Intuition from fractal geometry and real dynamics 
leads us to think that generically, a collision between the post-critical set and a hyperbolic 
repeller of dimension {\em larger than}  2 should yield massive bifurcation sets. 
 
 Here is a specific situation where we expect such a phenomenon to happen.
 
 \begin{conjecture}
 Let $f$ be a Lattès map on $\pd$. Then $f$ belongs to the closure of the interior of the bifurcation locus. 
 \end{conjecture}
 
 Indeed, a Lattès example is semi-conjugate to a multiplication on a complex 2-torus, so it admits hyperbolic repellers of dimension arbitrary close to 4. 
 In addition, it is conformal so the geometry of the perturbations of these repellers is expectedly easier to understand than in the general case. 
 Notice that Lattès mappings indeed belong to the bifurcation locus since the sum of the Lyapunov exponents of $\mu_f$ 
is minimal there (see \cite[Thm. 6.3]{bbd}).
  
  \medskip
  
  Finally, on a more ambitious note, one may ask whether the bifurcation locus is the closure of its interior in $\hold(\pd)$ or if on the contrary
   there are regions where the density of stability still holds. 
   As outlined in the introduction, bifurcations are created when a multiplier of a  repelling periodic orbit crosses the unit circle, so
an interesting approach to this problem would be  to understand when blenders are created in this process.

\end{document}